\documentclass[12pt,a4paper]{article}
\usepackage{amssymb,amsmath,amsthm}
\usepackage{latexsym}
\usepackage[mathscr]{eucal}
\usepackage{color}

\newtheorem{theorem}{Theorem}
\newtheorem{lemma}{Lemma}
\newtheorem{proposition}{Proposition}

\newtheorem{remark}{Remark}
\numberwithin{equation}{section}
\numberwithin{theorem}{section}
\numberwithin{lemma}{section}
\numberwithin{proposition}{section}
\numberwithin{corollary}{section}
\numberwithin{remark}{section}

\begin{document}
\title{Large time decay of the Oseen flow in exterior domains subject to the
Navier slip-with-friction boundary condition}
\author{Toshiaki Hishida \\
Graduate School of Mathematics \\
Nagoya University \\
Nagoya 464-8602, Japan \\
\texttt{hishida@math.nagoya-u.ac.jp} \\
}
\date{Dedicated to the memory of Professor Hermann Sohr}
\maketitle
\begin{abstract}
Consider the motion of a viscous incompressible fluid filling a 3D exterior domain $\Omega$ subject to 
the Navier slip-with-friction boundary condition as well as outflow at infinity. 
For the Oseen system as the linearization, we discuss the resolvent set under a certain relationship among 
the geometry of the boundary $\partial\Omega$, friction coefficient $\alpha(x)$ and the outflow $u_\infty$.
We then study the regularity of the resolvent near the origin in the complex plane to develop
$L^q$-$L^r$ decay estimates of the Oseen semigroup provided that  
$\alpha(x)+u_\infty\cdot\nu(x)/2\geq 0$ for every $x\in\partial\Omega$, 
where $\nu(x)$ stands for the outward unit normal to the boundary $\partial\Omega$.


\end{abstract}

\section{Introduction}
\label{intro}

Let us consider the Navier-Stokes flow in the exterior domain $\Omega$ of an obstacle in $\mathbb R^3$
with smooth boundary $\partial\Omega$ of class $C^{2,1}$ subject to the Navier slip-with-friction boundary condition,
while outflow $u_\infty\in \mathbb R^3$ (constant vector) is prescribed at spatial infinity:
\begin{equation}
\partial_tu+u\cdot\nabla u=\Delta u-\nabla p, \quad \mbox{div $u$}=0, 
\label{NS}
\end{equation}
\begin{equation}
\nu\cdot u|_{\partial\Omega}=0, \quad [(2\mathbb Du)\nu]_\tau+\alpha u_\tau|_{\partial\Omega}=0, \quad \lim_{|x|\to\infty}u=u_\infty,
\label{BC}
\end{equation}
where $u(x,t)\in\mathbb R^3$ and $p(x,t)\in\mathbb R$ are the velocity and pressure of the fluid,
\[
\mathbb Du=\frac{\nabla u+(\nabla u)^\top}{2} 
\]
is the symmetric gradient (deformation tensor) with $(\cdot)^\top$ being the transpose,
$\nu$ stands for the outward unit normal to $\partial\Omega$ and 
$\alpha(x)\geq 0$ denotes the friction coefficient at $x\in\partial\Omega$.
Given vector field $W$ at the boundary $\partial\Omega$, here and in what follows, its
tangential component $W_\tau$ is 
defined by
\[
W_\tau=W-(\nu\cdot W)\nu
=(\nu\times W)\times \nu.
\]

The condition $\mbox{\eqref{BC}}_2$ arises from
\begin{equation}
[\mathbb T(u,p)\nu]_\tau+\alpha u_\tau|_{\partial\Omega}=0
\label{BC1}
\end{equation}
with $\mathbb T(u,p)=2\mathbb Du-p\mathbb I$ being the Cauchy stress tensor, where $\mathbb I$ is
the unity matrix, nevertheless, it does not involve the pressure actually. 
When $\alpha>0$, we see from \eqref{BC1} that 
the tangential component of the 
velocity
is proportional to the one of the normal stress exerted by the fluid
with the proportional coefficient $1/\alpha$, which is called the slip length.
Since $(\nu\times (2\mathbb Du)\nu)\cdot\nu+\alpha (\nu\times u)\cdot\nu=0$,
the condition $\mbox{\eqref{BC}}_{1,2}$ is equivalent to
\begin{equation*}
\nu\cdot u|_{\partial\Omega}=0, \quad\nu\times(2\mathbb Du)\nu+\alpha\nu\times u|_{\partial\Omega}=0.
\end{equation*}
It is remarkable that the condition $\mbox{\eqref{BC}}_{1,2}$ can be reformulated in terms of the vorticity
at $\partial\Omega$ unlike the no-slip condition $u|_{\partial\Omega}=0$
which is widely adopted, see Chen and Qian \cite{cq10} for the details.
The particular case $\alpha=0$ is the full slip condition, whereas the no-slip one
is recovered formally when $\alpha\to\infty$.

There are some regimes in which the Navier condition would be suitable rather than the no-slip one.
One of them is related to
the occurrence of collision in a finite time of rigid bodies moving into a fluid within the context of fluid-structure interaction,
which is somehow impossible under the no-slip condition \cite{HT09}
in contrast to the case of the Navier one \cite{GHW15}.
Starting
from the work \cite{SoS73} by Solonnikov and Scadilov,
the Stokes and Navier-Stokes systems subject to the Navier boundary condition are
extensively studied from several 
points of view, see, for instance,
\cite{A21,A14,CLT24,FR16,FR19,Gi82,J25,KLT25,SS07,Shi07,SY08,W03}
and the references cited therein, however, to the best of our knowledge, 
the outflow $u_\infty$ is assumed to be zero for the exterior problem in the existing literature.

If we replace $u$ by $u_\infty+u$ in \eqref{NS}--\eqref{BC} 
to rewrite the system around $u_\infty$, 
eliminate the nonlinear term, replace the resulting inhomogeneous boundary 
condition
by the homogeneous one, 
set $\eta=-u_\infty$ and finally add the initial condition,
then we are led to the Oseen initial value problem 
\begin{equation}
\partial_tu=\Delta u+\eta\cdot\nabla u-\nabla p, \quad \mbox{div $u$}=0 \quad\mbox{in $\Omega\times (0,\infty)$},
\quad u(\cdot,0)=f,
\label{Os-iv}
\end{equation}
subject to
\begin{equation}
\nu\cdot u|_{\partial\Omega}=0, \quad 
[(2\mathbb Du)\nu]_\tau+\alpha u_\tau|_{\partial\Omega}=0, \quad
\lim_{|x|\to \infty}u=0.
\label{BC2}
\end{equation}
The resolvent system associated with \eqref{Os-iv}--\eqref{BC2} is given by
\begin{equation}
\lambda u-\Delta u-\eta\cdot\nabla u+\nabla p=f, \qquad \mbox{div $u$}=0,
\label{Os-re}
\end{equation}
subject to \eqref{BC2}.

The objective of the present paper is twofold.
Firstly, we discuss the resolvent set within the $L^q$-framework under the relationship \eqref{fric0} below among 
the geometry of $\Omega$, the friction $\alpha$ and $\eta\in\mathbb R^3\setminus\{0\}$.
In view of the symbol of the solution to \eqref{Os-re} in the whole space $\mathbb R^3$, see \eqref{formula},
one would expect that the resolvent set contains $\mathbb C\setminus S_\eta$ with
\begin{equation}
S_\eta:=\{\lambda\in\mathbb C;\; |\eta|^2\mbox{Re $\lambda$}+(\mbox{Im $\lambda$})^2\leq 0\}.
\label{os-spec-0}
\end{equation}
Indeed, under the no-slip boundary condition, it was shown by Kobayashi and Shibata \cite{KS98} that it is actually the case
for every $\eta\in\mathbb R^3\setminus\{0\}$.
It is interesting to ask whether or not the same thing holds under the Navier 
boundary condition.
In this paper, we will show that it is indeed true provided
\begin{equation}
\alpha(x)+\min\{\kappa_0(x),0\}\geq\frac{\eta\cdot\nu(x)}{2}
\label{fric0}
\end{equation}
for every $x\in\partial\Omega$, where
$\kappa_0(x)$ is the minimum of two principal curvatures of the boundary $\partial\Omega$
at $x$ in the direction of $\nu(x)$.
A key step for the proof is to show
the uniqueness of solutions (within a reasonable class) to \eqref{BC2}--\eqref{Os-re}
with $\lambda \in (\mathbb C\setminus S_\eta)\cup \{0\}$ under \eqref{fric0}.

The second objective is to deduce $L^q$-$L^r$ estimates of the solution, given by the Oseen semigroup, 
to \eqref{Os-iv}--\eqref{BC2}. 
As is well known, such decay estimates are useful to establish the global well-posedness of the Navier-Stokes
system \eqref{NS}--\eqref{BC} with initial data close to a steady state
in connection with its stability. 
If, in particular, the boundary condition is the full slip one $\alpha=0$, then
it turns out that
the only case when the desired large time decay is
available is the Stokes system $\eta=0$, for which the result was already shown by Shimada and Yamaguchi \cite{SY08}
(their argument works for general case $\alpha\geq 0$ as well, see \cite[Remark 1.4]{SY08}).
On the other side in which $\alpha\to\infty$ formally, large time decay of the Oseen semigroup subject
to the no-slip condition was well developed by Kobayashi and Shibata \cite{KS98}.
Our result may be regarded as 
a contribution that fills in between and recovers the result of \cite{SY08} for the Stokes semigroup as a particular case.
To be precise,
the desired $L^q$-$L^r$ estimates, in which the restriction on the summability exponents is the same
as in the no-slip case \cite{H20,H21,I89,KS98,MS98,S08} in 3D, are established in this paper under less condition
\begin{equation}
\alpha(x)\geq \frac{\eta\cdot\nu(x)}{2}
\label{fric00}
\end{equation}
for every $x\in\partial\Omega$ than \eqref{fric0}. 
The condition \eqref{fric00} allows us to show the uniqueness of solutions to \eqref{BC2}--\eqref{Os-re} for every $\lambda$
that lies in the right-half complex plane including the imaginary axis.
In fact, analysis of the resolvent for such $\lambda$,
especially near $\lambda=0$ along the imaginary axis,
is enough as well as crucial for the large time behavior although the
aforementioned result on the resolvent set ($\supset \mathbb C\setminus S_\eta$)
under \eqref{fric0} is of independent interest.

The proof is based on the local energy decay properties together with a cut-off procedure as in
\cite{ES04,ES05,H04,H16,H20,H21,HS09,Ma21,S08,SY08}; indeed, under the no-slip condition, this strategy is
traced back to Kobayashi and Shibata \cite{KS98}
for the Oseen system, that covers the Stokes one as well,
and, even before, to Iwashita \cite{I89} for the Stokes system.
Once we have the local energy decay estimate, on which we focus in this paper, then the subsequent procedure leads to the result
along the same argument as in the papers above (and thus may be omitted).
It is 
worthwhile
mentioning that there is 
another proof due to
Maremonti and Solonnikov \cite{MS98} for the Stokes semigroup subject to the no-slip condition.
In \cite{H18,H20} the present author developed an alternative route 
without spectral analysis to show the local energy decay, see \cite[Proposition 6.1]{H20},
for the non-autonomous system subject to the no-slip condition, and the subsequent procedure 
involves a nontrivial issue on the regularity
of $\partial_tu$ as well as the pressure (which corresponds to \cite[Theorem 1.4]{SS07}
under the Navier boundary condition) 
because of presence of rotation of the obstacle as well as the non-autonomous 
character.
For the Oseen semigroup under consideration, one
can get around such a regularity issue by making use of analyticity of the semigroup, see the references mentioned above.

The paper is organized as follows.
In the next section we present main results:
Theorem \ref{thm-0} on the resolvent set, Theorem \ref{thm} on the $L^q$-$L^r$ estimates.
After preparatory results in section \ref{pre}, section \ref{int} is devoted to analysis of the interior problem.
In section \ref{ext} we construct a parametrix of the resolvent in exterior domains and investigate its regularity near $\lambda =0$.
Theorem \ref{thm-0} is rephrased in Proposition \ref{prop-resol} and the proof is given there.
The local energy decay estimate is proved in the final section, which leads us to Theorem \ref{thm}.

\section{Results}
\label{result}

Let us begin with introducing notation.
Given a domain $G\subset\mathbb R^3$, $q\in [1,\infty]$ and integer $k\geq 0$, the standard Lebesgue and Sobolev spaces
are denoted by $L^q(G)$ and by $W^{k,q}(G)$.
When $q=2$, we write $H^k(G)=W^{k,2}(G)$.
We abbreviate the norm $\|\cdot\|_{q,G}=\|\cdot\|_{L^q(G)}$ and even $\|\cdot\|_q=\|\cdot\|_{q,\Omega}$, where
$\Omega$ is the exterior domain under consideration with $C^{2,1}$-boundary $\partial\Omega$.
Without loss, we assume that
\[
\mathbb R^3\setminus\Omega\subset B_1
\]
where $B_R$ denotes the open ball centered at the origin with
radius $R>0$.
We set $\Omega_R=\Omega\cap B_R$ for $R\geq 1$.
The class $C_0^\infty(G)$ consists of all $C^\infty$ functions with compact support in $G$, then $W^{k,q}_0(G)$
denotes the completion of $C_0^\infty(G)$ in $W^{k,q}(G)$, where $k>0$ is an integer.
In what follows we adopt the same symbols for denoting scalar and vector (even tensor) function spaces
as long as there is no confusion.

Let $X$ and $Y$ be two Banach spaces.
Then ${\mathcal L}(X,Y)$ stands for the Banach space consisting of all bounded linear operators from $X$ into $Y$.
We simply write ${\mathcal L}(X)={\mathcal L}(X,X)$.

We introduce the solenoidal function spaces over the exterior domain $\Omega$ with $C^{2,1}$-boundary $\partial\Omega$.
The class of $C_{0,\sigma}^\infty(\Omega)$ consists of all solenoidal vector fields being in $C_0^\infty(\Omega)$.
By $L^q_\sigma(\Omega)$ we denote the completion of $C_{0,\sigma}^\infty(\Omega)$ in $L^q(\Omega)$, then it is characterized as
\[
L^q_\sigma(\Omega)=\big\{u\in L^q(\Omega);\; \mbox{\rm div $u$}=0,\;\nu\cdot u|_{\partial\Omega}=0\big\}.
\]
The space of $L^q$-vector fields admits the Helmholtz decomposition
\[
L^q(\Omega)=L^q_\sigma(\Omega)\oplus \{\nabla p\in L^q(\Omega);\; p\in L^q_{\rm loc}(\overline{\Omega})\}
\]
which was proved by \cite{FM77, M82, SiSoh92}.
We denote by $P=P_q: L^q(\Omega)\to L^q_\sigma(\Omega)$ the Fujita-Kato projection 
associated
with the decomposition above.
Then we have $P_q^*=P_{q^\prime}$, where $1/q^\prime+1/q=1$.
Finally, several positive constants are denoted by $C$, which may change from line to line.

Let $1<q<\infty$, and
let us introduce the Stokes operator subject to the Navier 
boundary condition by
\begin{equation*}
\left\{
\begin{array}{l}
D(A)=\{u\in L^q_\sigma(\Omega)\cap W^{2,q}(\Omega);\; [(2\mathbb Du)\nu]_\tau+\alpha u_\tau=0\;\mbox{on $\partial\Omega$}\},  \\
Au=-P\Delta u,
\end{array}
\right.
\end{equation*}
where $\alpha\in C(\partial\Omega)$ is a given nonnegative function describing the friction at the boundary $\partial\Omega$.
If we denote by $A_q$ the operator $A$ acting on the space $L^q_\sigma(\Omega)$, 
we then see the duality relation $A_q^*=A_{q^\prime}$. 
Due to Shibata and Shimada \cite{SS07}, the operator $-A$ generates an analytic semigroup 
$\{e^{-tA}\}_{t\geq 0}$ of class
$(C_0)$ on the space $L^q_\sigma(\Omega)$
for every $q\in (1,\infty)$.
To be precise, 
for constant $\alpha\geq 0$,
they showed the following as well as
$\mathbb C\setminus (-\infty,0]\subset \rho(-A)$:
For every $\varepsilon\in (0,\frac{\pi}{2})$
and $\delta>0$,
there is a constant $C=C(\varepsilon,\delta)>0$ such that
$u=(\lambda+A)^{-1}f$ and the associated pressure $p=p(\lambda)$ 
enjoy
\begin{equation}
|\lambda|\|u\|_q 
+|\lambda|^{1/2}\|\nabla u\|_q 
+\|\nabla^2u\|_q 
+\|\nabla p\|_q\leq C\|f\|_q
\label{st-est}
\end{equation}
for all $f\in L^q_\sigma(\Omega)$ and $\lambda\in \Sigma_{\varepsilon,\delta}$, 
where
\[
\Sigma_{\varepsilon,\delta}:=\{\lambda\in \mathbb C;\; |\mbox{arg $\lambda$}|\leq \pi-\varepsilon,\, |\lambda|\geq \delta\}.
\]
Even for the case $\alpha\in C^1(\partial\Omega)$, the generation of the semigroup
is readily deduced as a corollary of their results:
in fact, thanks to the generalized resolvent estimate, see \cite[Theorem 1.3]{SS07},
for the 
system \eqref{Os-re} with $\eta=0$ subject to the full slip condition \eqref{BC2} ($\alpha=0$)
in which the homogeneous one is replaced by the inhomogeneous condition
$[(2\mathbb Du)\nu]_\tau=h$,
we apply their estimate to \eqref{BC2}--\eqref{Os-re} with $\alpha\in C^1(\partial\Omega)$ as well as $\eta=0$
by regarding $h=-\alpha u_\tau$, so that the desired a priori estimate \eqref{st-est} for large $|\lambda|$ is available.
Note that \eqref{st-est} 
implies
\begin{equation}
\|w\|_{W^{2,q}(\Omega)}\leq c\big(\|Aw\|_q+\|w\|_q\big) 
\label{st-elli}
\end{equation}
for every $w\in D(A)$ with some $c>0$.
In addition,
Shimada and Yamaguchi 
\cite{SY08} succeeded in deducing 
$L^q$-$L^r$ decay estimate \eqref{decay} below.
Although the full slip case $\alpha=0$ was discussed in \cite{SY08}, their argument works for 
constant
$\alpha\geq 0$.
The result is recovered in Theorem \ref{thm} as a particular case.

Given a constant vector $\eta\in\mathbb R^3$, 
the Oseen operator subject to the Navier 
boundary condition is now defined by
\[
D(L)=D(A), \qquad
Lu=L_\eta u=
-P[\Delta u+\eta\cdot\nabla u].
\]
Then we immediately find from \eqref{st-elli} together with the interpolation inequality that
\begin{equation}
\|w\|_{W^{2,q}(\Omega)}\leq c\big(\|Lw\|_q+\|w\|_q\big)
\label{os-elli}
\end{equation}
for every $w\in D(L)$ with some $c>0$.
Using \eqref{st-est} by means of the standard perturbation argument in terms of the Neumann series,
we readily see that, for every $\varepsilon\in (0,\frac{\pi}{2})$, there are constants $\lambda_0=\lambda_0(\varepsilon)>0$
and $C=C(\varepsilon)>0$ such that
\begin{equation}
\Sigma_{\varepsilon,\lambda_0}\subset\rho(-L)
\label{pre-resol}
\end{equation}
together with
\[
\|(\lambda+L)^{-1}f\|_q\leq 
\frac{C}{|\lambda|}\|f\|_q
\]
for all $f\in L^q_\sigma(\Omega)$ and $\lambda\in\Sigma_{\varepsilon,\lambda_0}$, which implies that the operator $-L$
generates an analytic semigroup $\{e^{-tL}\}_{t\geq 0}$
on $L^q_\sigma(\Omega)$ for every $q\in (1,\infty)$.

What interests us first of all is whether the spectral parameter $\lambda$ near the origin belongs to the resolvent set $\rho(-L)$.
Indeed, under the no-slip condition, we know from Kobayashi and Shibata \cite[Theorem 4.4]{KS98} 
that $\mathbb C\setminus S_\eta\subset \rho(-L)$
for every $\eta\in\mathbb R^3\setminus\{0\}$, where $S_\eta$ is given by \eqref{os-spec-0}.
There is further information for the no-slip case due to Farwig and Neustupa \cite[Theorem 1.2]{FN10}: 
$S_\eta$ is exactly the essential spectrum  
for all $\eta\in\mathbb R^3\setminus\{0\}$ and $q\in (1,\infty)$.
However, it does not seem to be always the case under the Navier boundary condition. 
In fact, in the following theorem, one needs the condition \eqref{ass-fric} or even \eqref{ass-fric0} in order that
the desired result is available.
\begin{theorem}
Suppose that a constant vector $\eta\in\mathbb R^3\setminus\{0\}$ 
and a nonnegative function $\alpha\in C(\partial\Omega)$ fulfill the relation
\begin{equation}
\alpha(x)\geq \frac{\eta\cdot\nu(x)}{2}
\label{ass-fric}
\end{equation}
for every $x\in\partial\Omega$, where $\nu(x)$ denotes the outward unit normal to 
the boundary $\partial\Omega\in C^{2,1}$.
Let $q\in (1,\infty)$, then we have
\begin{equation}
\overline{\mathbb C_+}\setminus\{0\}\subset \rho(-L) 
\label{resol-set-0}
\end{equation}
where $\overline{\mathbb C_+}:=\{\lambda\in\mathbb C;\; \mbox{\rm Re $\lambda$}\geq 0\}$.

For $x\in\partial\Omega$, let $\kappa(x)\leq 0$ be the least eigenvalue of the Weingarten map $-\nabla N$ of $\partial\Omega$
in the direction of $\nu$, where
$\kappa(x)$ is the minimum of two
principal curvatures of the boundary $\partial\Omega$ at $x$ in the direction of $\nu(x)$ if either of them is negative,
while $\kappa(x)=0$ if both of them are nonnegative,
see subsection \ref{wein}. 
Suppose, in addition, that $\eta\in\mathbb R^3\setminus\{0\}$ and $\alpha\in C(\partial\Omega)$ fulfill the relation
\begin{equation}
\alpha(x)+\kappa(x)\geq \frac{\eta\cdot\nu(x)}{2}
\label{ass-fric0}
\end{equation}
for every $x\in\partial\Omega$.
Then
\begin{equation}
\mathbb C\setminus S_\eta\subset \rho(-L)
\label{resol-set-1}
\end{equation}
holds true, where $S_\eta$ is given by \eqref{os-spec-0}.
\label{thm-0}
\end{theorem}

Theorem \ref{thm-0} will be rephrased in Proposition \ref{prop-resol} and the proof will be given there.

We turn to the $L^q$-$L^r$ estimates of the Oseen semigroup $e^{-tL}$.
Since the $L^q$-$L^r$ smoothing rate near $t=0$ is obvious by using
$\|Le^{-tL}\|_{{\mathcal L}(L^q_\sigma(\Omega))}\leq Ct^{-1}$ for $t\in (0,1)$ together with \eqref{os-elli} and interpolation inequalities,
the issue here is the rate of decay for $t\to\infty$.
This must be closely related to the picture of the resolvent set near $\lambda=0$ discussed in Theorem \ref{thm-0},
nevertheless, the desired situation \eqref{resol-set-1} is not necessary for our aim.
Indeed, with \eqref{resol-set-0} at hand
under the assumption \eqref{ass-fric}, one can proceed to analysis of the large time behavior.
\begin{theorem}
Suppose that a constant vector $\eta\in\mathbb R^3$ and 
a nonnegative function $\alpha\in C^1(\partial\Omega)$ fulfill the relation \eqref{ass-fric}
for every $x\in\partial\Omega$.
Let $q\in (1,\infty)$ and
\begin{equation*}
\left\{
\begin{array}{ll}
q\leq r\leq \infty\qquad &\mbox{for $j=0$}, \\
q\leq r<\infty &\mbox{for $j=1$}.
\end{array}
\right.
\end{equation*}
Then, 
for every $m>0$, there is a constant $C=C(m,\alpha,q,r,\Omega)>0$ such that
\begin{equation}
\|\nabla^je^{-tL}f\|_r\leq Ct^{-\frac{j}{2}-\frac{3}{2}(\frac{1}{q}-\frac{1}{r})}
\left\{
\begin{array}{ll}
\|f\|_q &\mbox{for $j=0$}, \\
(1+t)^{\max\{\frac{1}{2}-\frac{3}{2r},\,0\}}\|f\|_q \qquad&\mbox{for $j=1$},
\end{array}
\right.
\label{decay}
\end{equation}
for all $t>0$, $\eta\in\mathbb R^3$ with $|\eta|\leq m$, and $f\in L^q_\sigma(\Omega)$.
\label{thm}
\end{theorem}
\begin{remark}
Theorem \ref{thm} recovers the result of \cite{SY08} for the Stokes semigroup and the constant $C$ in \eqref{decay} can be
taken uniformly when $\eta\to 0$. 
On the other hand, the dependence of this constant $C$ on the friction $\alpha$ is not clear 
because of lack of 
information about 
such dependence in the a priori estimate 
for the Stokes system in bounded domains subject to the Navier 
boundary condition in spite of efforts by \cite{A21}, see Remark \ref{rem-fric}.
One may expect the uniformity for large $\alpha$ since the same result as in \eqref{decay}
holds true for the no-slip case \cite{KS98}.
\label{rem-main1}
\end{remark}
\begin{remark}
For the Stokes case $\eta=0$ subject to the no-slip condition, it is known from \cite{H11,MS98} that 
the rate of decay \eqref{decay} for $\nabla e^{-tA}$ is best possible.
This should be also the case under the Navier boundary condition as long as there exists a steady Stokes flow
with the forcing term $\mbox{\rm div $F$}$
subject to this boundary condition for which the total net force 
$\int_{\partial\Omega}\big(\mathbb T(u,p)+F\big)\nu\,d\sigma$ 
does not vanish.
This is because the coefficient of the leading term of the asymptotic representation at infinity of the Stokes flow is given by the net force
regardless of the boundary condition
\cite{H-b},
and because the optimal spatial decay of the steady Stokes flow
is closely related to the optimal temporal decay of $\nabla e^{-tA}$.
From this point of view, as pointed out by \cite[Section 5]{H11}, one can not claim the optimality of \eqref{decay} for the Oseen case
$\eta\in\mathbb R^3\setminus\{0\}$ since the steady Oseen flow possesses better spatial decay with wake structure.
\label{rem-main2}
\end{remark}

\section{Preparatory results}
\label{pre}

\subsection{Weingarten map}
\label{wein}

Let ${\mathcal S}$ be an orientable submanifold of class $C^2$ of codimension one in $\mathbb R^3$ with unit normal
$\nu=\nu(x)$, $x\in {\mathcal S}$.
In this subsection, following Duduchava, D. Mitrea and M. Mitrea \cite[Section 3]{RD06},
we briefly introduce an extended unit field 
and the Weingarten map of ${\mathcal S}$.

A vector field $N\in C^1({\mathcal U};\mathbb R^3)$ with ${\mathcal U}$ being a neighborhood of ${\mathcal S}$ is called an extended unit field for
${\mathcal S}$ if $N$ satisfies 
\begin{equation}
N|_{\mathcal S}=\nu, \qquad |N|=1, \qquad
(\nabla N)N=0 
\;\;\mbox{on ${\mathcal S}$}. 
\label{geo-ext1}
\end{equation}
According to the proof of \cite[Proposition 3.1]{RD06}, there exists actually an extended unit field for ${\mathcal S}$
although it is not unique.
Let us fix an extended unit field $N$.
Then, in addition to \eqref{geo-ext1}, the following properties hold true, see \cite[Proposition 3.4]{RD06},
where $-\nabla N$ is called the Weingarten map of ${\mathcal S}$ in the direction of $\nu$:
\medskip

(i)
$(\nabla N)^\top N=0$ in ${\mathcal U}$.
Thus, $(\nabla N)u|_{\mathcal S}$ is tangential to ${\mathcal S}$ for
any vector field $u: {\mathcal S}\to \mathbb R^3$.

(ii)
$(\nabla N)|_{\mathcal S}$ is independent of the choice of 
$N$ and depends only on ${\mathcal S}$.

(iii)
$(\nabla N)^\top=\nabla N$ on ${\mathcal S}$.

(iv)
The eigenvalues of $-\nabla N$ at 
$x\in {\mathcal S}$
consist of, besides zero,
two principal curvatures of ${\mathcal S}$ in the direction of $\nu(x)$.
\medskip

For later use, we show the following lemma.
\begin{lemma}
Let ${\mathcal S}\in C^2$ be as above with unit normal $\nu$ and
$N\in C^1({\mathcal U};\mathbb R^3)$ an extended unit 
field for ${\mathcal S}$, 
where ${\mathcal U}$ is a neighborhood of ${\mathcal S}$.
Suppose that a vector field
$u\in C^1({\mathcal U};\mathbb R^3)$ is tangential to ${\mathcal S}$, that is, $\nu\cdot u|_{\mathcal S}=0$.
Then we have
\begin{equation}
[(\nabla u)^\top\nu]_\tau=-[(\nabla N)u]_\tau=-(\nabla N)u 
\qquad\mbox{on ${\mathcal S}$}.
\label{wein-eq}
\end{equation}

Let $\Omega$ be an exterior domain in $\mathbb R^3$ with $C^2$-boundary ${\mathcal S}=\partial\Omega$ and
$\nu$ the outward unit normal to $\partial\Omega$.
Fix an extended unit field $N\in C^1({\mathcal U};\mathbb R^3)$  
as above, where the neighborhood ${\mathcal U}$ of $\partial\Omega$ may be assumed to be bounded.
Then, for every vector field $u\in H^2_{\rm loc}(\overline{\Omega})$ with $\nu\cdot u|_{\partial\Omega}=0$, we have
\eqref{wein-eq} in $L^2(\partial\Omega)$, where $\nabla u$ as well as $u$ is understood in the sense of trace.
\label{lem-wein}
\end{lemma}

\begin{proof}
We have
\begin{equation}
[N\times\nabla(N\cdot u)]\times N  \\
=[N\times \{(\nabla N)^\top u\}]\times N+[N\times\{(\nabla u)^\top N\}]\times N
\label{eq-near}
\end{equation}
in ${\mathcal U}$.
Since $\nu\cdot u|_{\mathcal S}=0$,
we see that the tangential derivative of $N\cdot u$ along ${\mathcal S}$ vanishes, that is,
$N\times \nabla(N\cdot u)=0$ on ${\mathcal S}$.
Hence, \eqref{eq-near} leads to
\begin{equation*}
[(\nabla N)^\top u]_\tau+[(\nabla u)^\top N]_\tau=0
\end{equation*}
on ${\mathcal S}$,
which combined with (iii) above implies the first equality of \eqref{wein-eq}. 
Successively, the second equality follows from (i).

As for the latter part, we observe \eqref{eq-near} in $H^1(\Omega\cap {\mathcal U})$ and, therefore, in
$L^2(\partial\Omega)$ as well.
We thus obtain \eqref{wein-eq} in $L^2(\partial\Omega)$.
The proof is complete.
\end{proof}

\subsection{Oseen resolvent in the whole space}
\label{os-wh}

Let us summarize useful regularity properties near $\lambda=0$ of the Oseen resolvent in the whole space $\mathbb R^3$
due to Kobayashi and Shibata \cite[Section 3]{KS98}.
See this literature for
the proof of several estimates 
in this subsection. 

It is enough to consider
$W^{2,q}$-estimate over a bounded domain, say, $B_3$ 
of the solution to 
the resolvent system
\begin{equation}
\lambda u-\Delta u-\eta\cdot\nabla u+\nabla p=f, \qquad
\mbox{div $u$}=0 \qquad\mbox{in $\mathbb R^3$}
\label{Os-wh}
\end{equation}
for the forcing term $f$ taken from the space
\begin{equation*}
L^q_{[R]}(\mathbb R^3):=\{f\in L^q(\mathbb R^3);\; f(x)=0\;\mbox{a.e. $\mathbb R^3\setminus B_R$}\},
\end{equation*}
where $R\geq 2$.
In terms of the Fourier multiplier operator, the solution is described as
\begin{equation}
u(\cdot,\lambda)=E_\eta(\lambda)f:=
{\mathscr F}^{-1}\left[\frac{|\xi|^2\mathbb I-\xi\otimes \xi}{(\lambda+|\xi|^2-i\eta\cdot\xi)|\xi|^2}\right]{\mathscr F}f
\label{formula}
\end{equation}
\begin{equation}
p=\Pi f:={\mathscr F}^{-1}\left[\frac{-i\xi}{|\xi|^2}\right]{\mathscr F}f
\label{formula-p}
\end{equation}
where $\mathscr F$ and ${\mathscr F}^{-1}$ stand for the Fourier transform and its inversion, respectively.
Note that the formula \eqref{formula} makes sense for both $\lambda=0$ and $\eta=0$ as well.
In view of the symbol of \eqref{formula}, when $\eta\in\mathbb R^3\setminus \{0\}$ (resp. $\eta=0$),
we see that $\lambda$ belongs to $\mathbb C\setminus S_\eta$
(resp. $\mathbb C\setminus (-\infty,0]$) if and only if
\[
\lambda+|\xi|^2-i\eta\cdot\xi\neq 0\qquad\forall\,\xi\in\mathbb R^3
\]
where $S_\eta$ is given by \eqref{os-spec-0}.
This implies that the spectrum for the whole space problem is contained in $S_\eta$ when $\eta\in\mathbb R^3\setminus\{0\}$.
By the Fourier multiplier theorem we find (\cite[Lemma 3.1]{KS98})
\begin{equation}
E_\eta(\lambda)\in {\mathcal L}(L^q(\mathbb R^3), W^{2,q}(\mathbb R^3))
\quad
\left\{
\begin{array}{ll}
\forall\,\lambda\in\mathbb C\setminus S_\eta &\mbox{for $\eta\in\mathbb R^3\setminus\{0\}$},  \\
\forall\,\lambda\in\mathbb C\setminus (-\infty,0]\quad &\mbox{for $\eta=0$},
\end{array}
\right. 
\label{fmt}
\end{equation}
\begin{equation}
\nabla\Pi\in {\mathcal L}(L^q(\mathbb R^3)).
\label{fmt-p}
\end{equation}
The following boundedness covers the case $\lambda=0$ as well, for which see Galdi \cite[IV.2, VII.4]{G-b}:
\begin{equation}
\begin{split}
\nabla^2E_\eta(\lambda)f\in {\mathcal L}(L^q(\mathbb R^3)), 
&\qquad
E_\eta(\lambda)\in {\mathcal L}(L^q(\mathbb R^3), W^{1,q}(B_\rho))   \\
&
\left\{
\begin{array}{ll}
\forall\,\lambda\in(\mathbb C\setminus S_\eta)\cup\{0\} &\mbox{for $\eta\in\mathbb R^3\setminus\{0\}$},  \\
\forall\,\lambda\in\mathbb C\setminus (-\infty,0)\quad &\mbox{for $\eta=0$},
\end{array}
\right.
\end{split}
\label{bdd-wh}
\end{equation}
for every $\rho>0$.

For the objective of this paper, estimates on the imaginary axis are particularly important. 
Let $1<q<\infty$ and set ${\mathcal E}_\eta(\tau):=\partial_\tau E_\eta(i\tau)$ for $\tau\in \mathbb R\setminus\{0\}$.
All the
estimates below
are deduced from the representation \eqref{formula}:
For every $m>0$, there is a constant $C=C(m,q,R)>0$ such that
\begin{equation}
\sup_{\lambda\in \overline{\mathbb C_+}}\|E_\eta(\lambda)\|_{{\mathcal L}(L^q_{[R]}(\mathbb R^3), W^{2,q}(B_3))}\leq C
\label{wh-1}
\end{equation}
\begin{equation}
\int_{-4}^4\|{\mathcal E}_\eta(\tau)\|_{{\mathcal L}(L^q_{[R]}(\mathbb R^3),W^{2,q}(B_3))}\,d\tau\leq C
\label{wh-2}
\end{equation}
\begin{equation}
\sup_{\tau\in\mathbb R}\|E_\eta(i(\tau+h))-E_\eta(i\tau)\|_{{\mathcal L}(L^q_{[R]}(\mathbb R^3),W^{2,q}(B_3))}\leq C|h|^{1/2}
\label{wh-3}
\end{equation}
\begin{equation}
\int_{-2}^2\|{\mathcal E}_\eta(\tau+h)-{\mathcal E}_\eta(\tau)\|_{{\mathcal L}(L^q_{[R]}(\mathbb R^3),W^{2,q}(B_3))}\,d\tau\leq C|h|^{1/2}
\label{wh-4}
\end{equation}
for all $\eta\in\mathbb R^3$ with $|\eta|\leq m$ and $h\in\mathbb R$ with $|h|\leq 1$
(\cite[Lemma 3.6]{KS98}).

As for the continuity with respect to $(\lambda,\eta)$,
we have the following:
For every $m>0$,
compact set $K\subset \overline{\mathbb C_+}$ 
and $\theta\in (0,\frac{1}{2})$, there is a constant $C=C(m,K,\theta,q,R)>0$ such that
\begin{equation}
\|E_\eta(\lambda)-E_{\eta^\prime}(\lambda^\prime)\|_{{\mathcal L}(L^q_{[R]}(\mathbb R^3),W^{2,q}(B_3))}
\leq C\big(|\lambda-\lambda^\prime|+|\eta-\eta^\prime|\big)^\theta
\label{wh-conti}
\end{equation}
for all $(\lambda,\eta),\, (\lambda^\prime,\eta^\prime)\in K\times \overline{B_m}$ (\cite[Lemmas 3.3, 3.4]{KS98}).

\section{Interior problem}
\label{int}

Let $D$ be a bounded domain in $\mathbb R^3$ with $C^{2,1}$-boundary $\partial D$.
This section studies the interior problem for the Oseen resolvent system subject to the Navier 
boundary condition
\begin{equation}
\left\{
\begin{array}{ll}
\lambda u-\Delta u-\eta\cdot\nabla u+\nabla p=f, \quad \mbox{div $u$}=0 \;\;&\mbox{in $D$},   \\
\nu\cdot u=0, \qquad 
[2\mathbb D(u)\nu]_\tau+\alpha u_\tau=0 &\mbox{on $\partial D$},
\end{array}
\right.
\label{Os-int}
\end{equation}

We fix a subdomain $D_0\subset D$ with $|D_0|>0$ and single out a solution to \eqref{Os-int} in such a way that
\begin{equation}
\int_{D_0}p(x)\,dx=0.
\label{ave-zero}
\end{equation}
We begin with uniqueness of solutions to \eqref{Os-int}, see 
Proposition \ref{int-uni} below, where we take 
into account the following fact
found in, for instance, 
\cite{A14,KLT25,W03}; 
especially, the proof is given by \cite{KLT25} in detail.
\begin{lemma}
[{\cite[Proposition 7.2]{KLT25}}]
Let $D$ be a bounded domain as above.
Then the space 
\[
\{u\in H^1(D);\; \mathbb Du=O,\,
\nu\cdot u|_{\partial D}=0\}
\]
is nontrivial if and only if $D$ is
axisymmetric about an axis 
$s\mapsto a+sb,\, s\in\mathbb R$,
with 
some $b\in \mathbb R^3\setminus\{0\}$ and $a\in\mathbb R^3$.
\label{auxi}
\end{lemma}
\begin{proposition}
Suppose that a constant vector $\eta\in\mathbb R^3$ and a nonnegative function $\alpha\in C(\partial D)$ fulfill the relation
\begin{equation}
\alpha(x)\geq\frac{\eta\cdot\nu(x)}{2}
\label{ass-fric2}
\end{equation}
for every $x\in \partial D$, where $\nu(x)$ denotes the outward unit normal to the boundary $\partial D\in C^{2,1}$.
If, in particular, $\alpha\equiv 0$, it is additionally assumed that $D$ is not axisymmetric about any axis. 
Let $q\in (1,\infty)$ and
\begin{equation*}
\begin{array}{ll}
\lambda\in\overline{\mathbb C_+} &\mbox{for $\eta\in\mathbb R^3\setminus \{0\}$},  \\
\lambda\in\mathbb C\setminus (-\infty,0) \qquad&\mbox{for $\eta=0$}.
\end{array}
\end{equation*}
Then the only solution $(u,p)\in W^{2,q}(D)\times W^{1,q}(D)$ to \eqref{Os-int}--\eqref{ave-zero} 
with $f=0$ is $(u,p)=(0,0)$.

For $x\in \partial D$, let
$\kappa(x)\leq 0$ be the least eigenvalue of the Weingarten map  $-\nabla N$
of $\partial D$ in the direction of $\nu$
(see subsection \ref{wein}).
Suppose in addition that $\eta\in\mathbb R^3\setminus\{0\}$ and $\alpha\in C(\partial D)$ fulfill the relation
\begin{equation}
\alpha(x)+\kappa(x)\geq\frac{\eta\cdot\nu(x)}{2}
\label{ass-fric3}
\end{equation}
for every $x\in\partial D$.
Let $\lambda\in (\mathbb C\setminus S_\eta)\cup \{0\}$ with 
$S_\eta$ being given by \eqref{os-spec-0}.
Then the same uniqueness assertion above holds true.
\label{int-uni}
\end{proposition}

\begin{proof}
First of all, we observe $u\in H^2(D)$ and $p\in H^1(D)$ (even though $q$ is close to $1$) by bootstrap argument with the aid of the
regularity theory for the Stokes system subject to the Navier 
boundary condition \cite{A21,A14,SS07}; indeed, less condition on the friction $\alpha(x)$ than ours is imposed in \cite{A21}.
We multiply the equation by $\overline{u}$, integrate and use the boundary
condition to get
\begin{equation}
\lambda \|u\|_{2,D}^2+2\|\mathbb Du\|_{2,D}^2+\int_{\partial D}\alpha|u|^2\,d\sigma
-\int_D(\eta\cdot\nabla u)\cdot\overline{u}\,dx=0.
\label{energy-1}
\end{equation}
In fact, the third term arises from
\[
\int_{\partial D}[\mathbb T(u,p)\nu]\cdot\overline{u}\,d\sigma
=\int_{\partial D}[\mathbb T(u,p)\nu]_\tau\cdot \overline{u}_\tau\,d\sigma
=-\int_{\partial D}\alpha |u_\tau|^2\,d\sigma
=-\int_{\partial D}\alpha |u|^2\,d\sigma
\]
by taking into account that the boundary condition in \eqref{Os-int} comes from \eqref{BC1}.
Since
\[
\int_D [(\eta\cdot\nabla u)\cdot\overline{u}+u\cdot(\eta\cdot\nabla\overline{u})]\,d\sigma
=\int_D\mbox{div $(\eta|u|^2)$}\,dx
=\int_{\partial D}\eta\cdot\nu|u|^2\,d\sigma,
\]
the real and imaginary parts respectively give
\begin{equation}
(\mbox{Re $\lambda$})\|u\|_{2,D}^2+2\|\mathbb Du\|_{2,D}^2+\int_{\partial D}\left(\alpha-\frac{\eta\cdot\nu}{2}\right)|u|^2\,d\sigma=0,
\label{ene-re}
\end{equation}
\begin{equation}
(\mbox{Im $\lambda$})\|u\|_{2,D}^2-\mbox{Im}\int_D(\eta\cdot\nabla u)\cdot\overline{u}\,dx=0.
\label{ene-im}
\end{equation}
We immediately see from \eqref{ass-fric2} that $u=0$ if $\mbox{Re $\lambda$}>0$. 
When $\mbox{Re $\lambda$}=0$, we have 
$\mathbb Du=O$.
Then the rigid motion satisfying $\nu\cdot u=0$ as well as $\alpha \nu\times u=0$ 
at $\partial D$ should be $u=0$ unless $\alpha$ is identically zero.
When $\alpha\equiv 0$, due to Lemma \ref{auxi},
the only case in which a nontrivial rigid motion $u$ with $\nu\cdot u=0$ at $\partial D$
is available is that $D$ is axisymmetric about an axis.
This case is ruled out by the assumption.
In this way, we are led to $u=0$ for $\lambda\in \overline{\mathbb C_+}$, and thereby $\nabla p=0$, yielding $p=0$ by \eqref{ave-zero}.
If, in particular, $\eta=0$, then \eqref{ene-im} implies $u=0$ for $\lambda\in\mathbb C\setminus\mathbb R$ as well.

We next consider the case under further condition \eqref{ass-fric3} when $\eta\in\mathbb R^3\setminus\{0\}$.
Since $\mbox{div $u$}=0$, we have $\Delta u=\mbox{div $(2\mathbb Du)$}$, that together with $\nu\cdot u=0$ at $\partial D$ leads to
\begin{equation}
\begin{split}
2\|\mathbb Du\|_{2,D}^2
&=\|\nabla u\|_{2,D}^2+\int_{\partial D}[(\nabla u)^\top\nu]\cdot\overline{u}\,d\sigma  \\
&=\|\nabla u\|_{2,D}^2-\int_{\partial D}[(\nabla N)u]\cdot\overline{u}\,d\sigma
\end{split}
\label{geo}
\end{equation}
by Lemma \ref{lem-wein}, which implies that
\begin{equation}
2\|\mathbb Du\|_{2,D}^2\geq \|\nabla u\|_{2,D}^2
+\int_{\partial D}\kappa|u|^2\,d\sigma.
\label{geo2}
\end{equation}
Combining \eqref{ene-re} with \eqref{geo2} gives
\begin{equation}
(\mbox{Re $\lambda$})\|u\|_{2,D}^2+\|\nabla u\|_{2,D}^2
+\int_{\partial D}\left(\alpha+\kappa-\frac{\eta\cdot\nu}{2}\right)|u|^2\,d\sigma\leq 0.
\label{ene-re2}
\end{equation}
By \eqref{ass-fric3} we at once find that $u=0$ when $\mbox{Re $\lambda$}\geq 0$.
Consider the case $\mbox{Re $\lambda$}<0$ by using \eqref{ene-im}.
Then we have
\[
(\mbox{Im $\lambda$})^2\|u\|_{2,D}^4\leq |\eta|^2\|\nabla u\|_{2,D}^2\|u\|_{2,D}^2
\leq -|\eta|^2(\mbox{Re $\lambda$})\,\|u\|^4_{2,D},
\]
from which we find that $\lambda\in \mathbb C\setminus S_\eta$ leads to $u=0$.
The proof is complete.
\end{proof}

The following proposition provides a solution operator to \eqref{Os-int}--\eqref{ave-zero} along with regularity properties.
\begin{proposition}
Under the same assumptions of the first half of Proposition \ref{int-uni}, let
\[
\begin{array}{ll}
\lambda \in\overline{\mathbb C_+} &\mbox{for $\eta\in\mathbb R^3\setminus\{0\}$},  \\
\lambda\in\mathbb C\setminus (-\infty,0)\qquad &\mbox{for $\eta=0$}.
\end{array}
\]
Then there exist bounded operators $M_\eta(\lambda)$ and $N_\eta(\lambda)$ from $L^q(D)$ into
$W^{2,q}(D)$ and $W^{1,q}(D)$, respectively,
such that the pair $\big(M_\eta(\lambda)f,N_\eta(\lambda)f\big)$ gives a unique solution
of \eqref{Os-int} subject to \eqref{ave-zero} for all $f\in L^q(D)$, and it is analytic in a certain open neighborhood of
$\overline{\mathbb C_+}$ (resp. $\mathbb C\setminus (-\infty,-\rho]$ for some $\rho>0$) when $\eta\in\mathbb R^3\setminus\{0\}$
(resp. $\eta=0$).

Instead, under the same assumptions of the second half of Proposition \ref{int-uni}, let
$\lambda\in (\mathbb C\setminus S_\eta)\cup\{0\}$.
Then the same conclusion as above holds true, and $\big(M_\eta(\lambda)f,N_\eta(\lambda)f\big)$ is analytic in 
$(\mathbb C\setminus S_\eta)\cup \{|\lambda|<\rho\}$ for some $\rho >0$.

Moreover, given $m>0$ and compact set $K$ satisfying
\begin{equation}
\begin{array}{ll}
K\subset \overline{\mathbb C_+}&\mbox{under \eqref{ass-fric2}}, \\
K\subset\{\lambda\in\mathbb C;\; m^2\mbox{\rm Re $\lambda$}+(\mbox{\rm Im $\lambda$})^2>0\}\cup\{0\}\;\;&
\mbox{under \eqref{ass-fric3}},
\end{array}
\label{cpt-K}
\end{equation}
the solution enjoys the following properties.

\begin{enumerate}
\item
For every integer $j\geq 0$, there is a constant $C=C(j,m,K,\alpha,q,D)>0$ 
such that
\begin{equation}
\|\partial^j_\lambda M_\eta(\lambda)f\|_{W^{2,q}(D)}+\|\partial^j_\lambda N_\eta(\lambda)f\|_{W^{1,q}(D)}\leq C\|f\|_{q,D}
\label{high-int}
\end{equation}
for all $(\lambda,\eta)\in K\times \overline{B_m}$ and $f\in L^q(D)$.

\item
There is a constant $C=C(m,K,\alpha,q,D)>0$ such that
\begin{equation}
\|M_\eta(\lambda)f-M_{\eta^\prime}(\lambda^\prime)f\|_{W^{2,q}(D)}
+\|N_\eta(\lambda)f-N_{\eta^\prime}(\lambda^\prime)f\|_{W^{1,q}(D)}
\leq C\big(|\lambda-\lambda^\prime|+|\eta-\eta^\prime|\big)\|f\|_q
\label{conti-int}
\end{equation}
for all $(\lambda,\eta),\, (\lambda^\prime,\eta^\prime)\in K\times \overline{B_m}$ and $f\in L^q(D)$.
\end{enumerate}
\label{prop-int}
\end{proposition}

\begin{proof}
For the Stokes system \eqref{Os-int}--\eqref{ave-zero} with $(\lambda,\eta)=(0,0)$,
it follows from  
\cite{A21,A14,SS07} that  
there is a unique solution
\[
u=M_0(0)f, \qquad p=N_0(0)f
\]
such that
\begin{equation}
\|M_0(0)f\|_{W^{2,q}(D)}+\|N_0(0)f\|_{W^{1,q}(D)}\leq C\|f\|_{q,D}.
\label{st-int}
\end{equation}
We then see that $M_0(0)$ and $\nabla M_0(0)$ are compact operators from $L^q(D)$ into itself by the Rellich theorem.
Let us look for a solution to \eqref{Os-int}--\eqref{ave-zero} of the form $(u,p)=(M_0(0)g,N_0(0)g)$ with a suitable $g\in L^q(D)$:
\[
\lambda u-\Delta u-\eta\cdot\nabla u+\nabla p=g+\lambda M_0(0)g-\eta\cdot\nabla M_0(0)g, \qquad \mbox{div $u$}=0
\]
together with the boundary condition in \eqref{Os-int}.
Then the operator $1+\lambda M_0(0)-\eta\cdot\nabla M_0(0)$
is injective and, thereby,  
invertible in $L^q(D)$ by the Fredholm alternative.
In fact, let $g\in L^q(D)$ satisfy $g+\lambda M_0(0)g-\eta\cdot\nabla M_0(0)g=0$, then
the pair $(u,p)=(M_0(0)g,N_0(0)g)\in W^{2,q}(D)\times W^{1,q}(D)$ must be the trivial one by Proposition \ref{int-uni}, yielding 
$g=-\Delta u+\nabla p=0$. 
We thus find that the pair
\begin{equation}
\begin{split}
&u=M_\eta(\lambda)f:=M_0(0)[1+\lambda M_0(0)-\eta\cdot\nabla M_0(0)]^{-1}f\in W^{2,q}(D) \\
&p=N_\eta(\lambda)f:=N_0(0)[1+\lambda M_0(0)-\eta\cdot\nabla M_0(0)]^{-1}f\in W^{1,q}(D)
\end{split}
\label{sol-int}
\end{equation}
provides a unique solution to \eqref{Os-int} for all 
$\lambda \in\overline{\mathbb C_+}$ (resp. $\lambda\in (\mathbb C\setminus S_\eta)\cup\{0\}$)
$\eta\in\mathbb R^3$ and $f\in L^q(D)$ for the first case (resp. second case).

In what follows, we will describe the proof for the first case under the condition \eqref{ass-fric2}.
We fix $m>0$ arbitrarily.
Since $(\lambda,\eta)\mapsto 1+\lambda M_0(0)-\eta\nabla M_0(0)$ is continuous from 
$\overline{\mathbb C_+}\times \overline{B_m}$ to ${\mathcal L}(L^q(D))$,
so is $(\lambda,\eta)\mapsto [1+\lambda M_0(0)-\eta\cdot\nabla M_0(0)]^{-1}$.
In fact, if $(\lambda,\eta)$ and $(\lambda^\prime,\eta^\prime)$ 
satisfy
\begin{equation*}
\begin{split}
&|\lambda^\prime-\lambda|\|M_0(0)\|_{{\mathcal L}(L^q(D))}+|\eta^\prime-\eta|\|\nabla M_0(0)\|_{{\mathcal L}(L^q(D))}  \\
&\qquad\leq \frac{1}{2\big\|[1+\lambda M_0(0)-\eta\cdot\nabla M_0(0)]^{-1}\big\|_{{\mathcal L}(L^q(D))}}
\end{split}
\end{equation*}
as well as lie in $\overline{\mathbb C_+}\times\overline{B_m}$, then
we are led to the 
Neumann series representation 
\begin{equation}
\begin{split}
&[1+\lambda^\prime M_0(0)-\eta^\prime\cdot\nabla M_0(0)]^{-1}  \\
&=[1+\lambda M_0(0)-\eta\cdot\nabla M_0(0)]^{-1}  \\
&\quad\sum_{k=0}^\infty\Big(-\big\{(\lambda^\prime-\lambda)M_0(0)-(\eta^\prime-\eta)\cdot\nabla M_0(0)\big\}
[1+\lambda M_0(0)-\eta\cdot\nabla M_0(0)]^{-1}\Big)^k
\end{split}
\label{int-neu}
\end{equation}
which implies the continuity of
$(\lambda,\eta)\mapsto [1+\lambda M_0(0)-\eta\cdot\nabla M_0(0)]^{-1}$; to be precise,
\begin{equation}
\begin{split}
&\quad \big\|[1+\lambda^\prime M_0(0)-\eta^\prime\cdot\nabla M_0(0)]^{-1}-[1+\lambda M_0(0)-\eta\cdot\nabla M_0(0)]^{-1}\big\|_{{\mathcal L}(L^q(D))}  \\
&\leq 2\big\|[1+\lambda M_0(0)-\eta\cdot\nabla M_0(0)]^{-1}\big\|^2_{{\mathcal L}(L^q(D))}
\big(|\lambda^\prime-\lambda|\|M_0(0)\|_{{\mathcal L}(L^q(D))}+|\eta^\prime-\eta|\|\nabla M_0(0)\|_{{\mathcal L}(L^q(D))}\big).
\end{split}
\label{conti-inv}
\end{equation}
Hence, we have 
\begin{equation}
\sup_{(\lambda,\eta)\in K\times \overline{B_m}}\big\|[1+\lambda M_0(0)-\eta\cdot\nabla M_0(0)]^{-1}\big\|_{{\mathcal L}(L^q(D))}<\infty
\label{K-bdd}
\end{equation}
for each compact set $K\subset \overline{\mathbb C_+}$, which combined with \eqref{st-int}
gives \eqref{high-int} with $j=0$ in view of \eqref{sol-int}.
From \eqref{conti-inv} and \eqref{K-bdd} we immediately obtain \eqref{conti-int}.

Furthermore, we fix $\eta\in\mathbb R^3$ and set $\eta^\prime=\eta$ in \eqref{int-neu}, that tells us that
$\lambda\mapsto [1+\lambda M_0(0)-\eta\cdot\nabla M_0(0)]^{-1}$ is analytic in a ceratin open neighborhood of
$\overline{\mathbb C_+}$ with values in ${\mathcal L}(L^q(D))$ and, therefore, so is 
$\lambda\mapsto (M_\eta(\lambda),N_\eta(\lambda))$ with values in ${\mathcal L}(L^q(D), W^{2,q}(D)\times W^{1,q}(D))$.

It remains to show \eqref{high-int} for every $j\geq 1$.
By taking the differentiation of \eqref{Os-int}, we see that 
$(\partial_\lambda M_\eta(\lambda)f,\partial_\lambda N_\eta(\lambda)f)\in W^{2,q}(D)\times W^{1,q}(D)$
is a solution to \eqref{Os-int}--\eqref{ave-zero} with $f$ replaced by $-M_\eta(\lambda)f$.
By uniqueness of solutions (Proposition \ref{int-uni}), we infer
\[
\partial_\lambda M_\eta(\lambda)f=-M_\eta(\lambda)^2f, \qquad
\partial_\lambda N_\eta(\lambda)f=-N_\eta(\lambda)M_\eta(\lambda)f.
\]
By induction we find
\[
\partial_\lambda^jM_\eta(\lambda)f=(-1)^j
j!M_\eta(\lambda)^{j+1}f, \qquad
\partial_\lambda^jN_\eta(\lambda)f=(-1)^j
j!N_\eta(\lambda)M_\eta(\lambda)^jf
\]
which leads to \eqref{high-int} for every $j\geq 1$ from \eqref{high-int} with $j=0$.

For the second case under the condition \eqref{ass-fric3}, 
the argument above with $K$ as in \eqref{cpt-K} leads us to the desired conclusion.
The proof is complete.
\end{proof}
\begin{remark}
In \eqref{high-int},
we are interested in the dependence of the constant $C$ on the friction $\alpha$, which determines the dependence of
the one in \eqref{decay} via \eqref{key-1}--\eqref{key-2} and \eqref{led} below.
For the Stokes system \eqref{Os-int}--\eqref{ave-zero} with $(\lambda,\eta)=(0,0)$,
the constant of the a priori estimate
of the lower order norm $\|u\|_{W^{1,q}(D)}+\|p\|_{q,D}$ is independent of the friction $\alpha$
as long as $D$ is non-axisymmetric, while it is uniform only for large $\alpha$ if $D$ is axisymmetric,
see \cite[Theorem 6.11]{A21}; indeed, non-uniformity near $\alpha=0$ in the latter case should be related to Lemma \ref{auxi}.
Even for the non-axisymmetric domain $D$, however, it seems to be still unclear whether the constant in \eqref{st-int}
is independent of $\alpha$.
\label{rem-fric}
\end{remark}

\section{Oseen resolvent in exterior domains}
\label{ext}

In this section we construct a parametrix of solutions to 
the Oseen resolvent system in the exterior domain $\Omega$ subject to the Navier 
boundary condition on
$\partial\Omega\in C^{2,1}$:
\begin{equation}
\begin{array}{ll}
\lambda u-\Delta u-\eta\cdot\nabla u+\nabla p=f, \qquad \mbox{div $u$}=0\;\;&\mbox{in $\Omega$}, \\
\nu\cdot u=0, \qquad [(2\mathbb Du)\nu]_\tau+\alpha u_\tau=0 &\mbox{on $\partial\Omega$}.
\end{array}
\label{Os-resol}
\end{equation}
The boundary condition at infinity is later taken into account in the sense of summability,
see \eqref{sum-weak} and \eqref{para-sum}.

In order to study the uniqueness of solutions to \eqref{Os-resol}, 
let us introduce the following lemmas, both of which connect $\mathbb Du$ with $\nabla u$.
The first one is 
Korn's first inequality in exterior domains,
that holds true without any boundary condition such as
$\nu\cdot u=0$ at $\partial\Omega$ nor $\mbox{div $u$}=0$ in $\Omega$. 
The result is due to Shibata and Soga \cite{SSo89}.
See also Ito \cite[Theorem 6.5]{It90} for an alternative proof.
Since the constant $c_0$ in Korn's inequality \eqref{korn} is involved in \eqref{os-spec2} below under less assumption
\eqref{ass-fric}, we are interested in the best constant.
Indeed, Ito \cite{It90} discussed this issue in detail for the half-space problem, however, not for the exterior problem.
\begin{lemma}
[{\cite[Theorem 1.5]{SSo89}}]
There is a constant $c_0=c_0(\Omega)>0$ such that
\begin{equation}
\|\nabla u\|_2^2\leq c_0\|\mathbb Du\|_2^2
\label{korn}
\end{equation}
for all $u\in \widehat H^1(\Omega)^3$, which is defined as the completion of $C_0^\infty(\overline{\Omega})^3$ with respect to
the norm $\|\nabla(\cdot)\|_2$.
\label{lem-korn}
\end{lemma}

The second one is described in terms of
the Weingarten map of $\partial\Omega$ introduced in subsection \ref{wein}.
This suggests that the geometry of $\partial\Omega$ is involved behind the relation between $\mathbb Du$ and $\nabla u$.
The result can be found in \cite{J25} by Jin in slightly a different form.
\begin{lemma}
[{\cite[Lemma 4.2]{J25}}]
Let $u\in H^1(\Omega)\cap H^2_{\rm loc}(\overline{\Omega})$ be a complex valued vector field
satisfying $\mbox{\rm div $u$}=0$ and $\nu\cdot u|_{\partial\Omega}=0$.
Then we have
\begin{equation}
2\|\mathbb Du\|_2^2=\|\nabla u\|_2^2-\int_{\partial\Omega}[(\nabla N)u]\cdot \overline{u}\,d\sigma,
\label{geo-ext}
\end{equation}
where $-\nabla N$ is the Weingarten map of $\partial\Omega$ in the direction of $\nu$.
\label{Du-wein}
\end{lemma}

\begin{proof}
The proof is essentially the same as in deduction of \eqref{geo} in bounded domains.
To justify the computation in exterior domains, nevertheless, 
using a cut-off function
$\phi\in C_0^\infty(B_2;[0,1])$ such that $\phi=1$ in $B_1$,
we multiply $\Delta u=\mbox{div $(2\mathbb Du)$}$ by $\phi_\rho \overline{u}$ with $\phi_\rho(x)=\phi(x/\rho)$ and then integrate
the resulting equality
to find 
\[
2\int_\Omega |\mathbb Du|^2\phi_\rho\,dx
=\int_\Omega |\nabla u|^2\phi_\rho\,dx
+\int_{\partial\Omega}[(\nabla u)^\top\nu]\cdot \overline{u}\,d\sigma
-\int_{A_\rho}[(\nabla u)^\top\nabla\phi_\rho]\cdot \overline{u}\,dx,
\]
where $A_\rho=\{\rho<|x|<2\rho\}$.
Since $|u||\nabla u|\in L^{3/2}(\Omega)$ and since $\|\nabla \phi_\rho\|_3$ is independent of $\rho$, letting $\rho\to\infty$ leads us to
\eqref{geo-ext} by taking into account Lemma \ref{lem-wein}.
\end{proof}.
\begin{proposition}
Suppose that a constant vector $\eta\in\mathbb R^3$ and a nonnegative function $\alpha\in C(\partial\Omega)$
fulfill the relation \eqref{ass-fric} for every $x\in\partial\Omega$, where $\nu(x)$ denotes the outward unit normal to
the boundary $\partial\Omega\in C^{2,1}$.
Let $q,\, r,\, s\in (1,\infty)$ and
\begin{equation*}
\begin{array}{ll}
\lambda\in (\mathbb C\setminus \widetilde S_\eta)\cup\{0\}\qquad &\mbox{for $\eta\in\mathbb R^3\setminus \{0\}$}, \\
\lambda\in \mathbb C\setminus (-\infty,0) &\mbox{for $\eta=0$},
\end{array}
\end{equation*}
with
\begin{equation}
\widetilde S_\eta:=\big\{\lambda\in\mathbb C;\; \frac{c_0}{2}|\eta|^2\mbox{\rm Re $\lambda$}+(\mbox{\rm Im $\lambda$})^2\leq 0\big\},
\label{os-spec2}
\end{equation}
where $c_0$ is the constant in \eqref{korn}.
Then the only solution $(u,p)\in W^{2,q}_{\rm loc}(\overline{\Omega})\times W^{1,q}_{\rm loc}(\overline{\Omega})$
to \eqref{Os-resol} 
with $f=0$ satisfying
\begin{equation}
u\in L^r(\Omega), \qquad \nabla p\in L^s(\Omega)
\label{sum-weak}
\end{equation}
is $(u,p)=(0,p_\infty)$ with some constant $p_\infty\in\mathbb C$.

For $x\in\partial\Omega$, let $\kappa(x)\leq 0$ be the least eigenvalue of the Weingarten map $-\nabla N$ of $\partial\Omega$
in the direction of $\nu$
(see subsection \ref{wein}).
In addition to the conditions above, suppose that 
$\eta\in\mathbb R^3\setminus\{0\}$ and $\alpha\in C(\partial\Omega)$ fulfill the relation
\eqref{ass-fric0} for every $x\in\partial\Omega$.
Let $\lambda\in (\mathbb C\setminus S_\eta)\cup\{0\}$, where $S_\eta$ is given by \eqref{os-spec-0}.
Then the same uniqueness assertion above holds true.
\label{ext-uni}
\end{proposition}

\begin{proof}
By the regularity theory for the Stokes system subject to the Navier 
boundary condition
\cite{A21,A14,SS07} together with a cut-off procedure, we may assume that
$(u,p)\in H^2_{\rm loc}(\overline{\Omega})\times H^1_{\rm loc}(\overline{\Omega})$ even though $q$ is close to $1$.
On the other hand, weak summability \eqref{sum-weak} implies that $(u,p-p_\infty)$ with some constant $p_\infty\in\mathbb C$ 
behaves like the fundamental solution to the Oseen resolvent system in the whole space $\mathbb R^3$.
Even for the worst case $\lambda=0$ we have
\begin{equation}
\begin{split}
&(\nabla u,p-p_\infty)\in L^{q_0}(\Omega)\quad\forall\,q_0\in (3/2,\infty), \\
&u\in L^{r_0}(\Omega)\quad\left\{
\begin{array}{ll}
\forall\,r_0\in (2,\infty)\quad&\mbox{for $\eta\in\mathbb R^3\setminus\{0\}$},  \\
\forall\,r_0\in (3,\infty)&\mbox{for $\eta=0$}.
\end{array}
\right. 
\end{split}
\label{improved-sum}
\end{equation}
This is verified by considering the equation in $\mathbb R^3$ that the pair $((1-\phi)u+w,(1-\phi)p)$ obeys,
where $\phi$ is a suitable cut-off function and $w$ is a correction term 
(to be compactly supported since $\nu\cdot u=0$ at $\partial\Omega$)
to recover the solenoidal condition.
For the details of this argument, see for instance \cite[Theorem 1]{H-b}.

Let $\phi_\rho$ be the same cut-off function as in the proof of Lemma \ref{Du-wein}.
We multiply the equation
\[
\lambda u-\Delta u-\eta\cdot\nabla u+\nabla (p-p_\infty)=0
\]
by $\phi_\rho\overline{u}$, integrate and use the boundary condition  
to furnish
\begin{equation}
\begin{split}
&\lambda\int_\Omega\phi_\rho|u|^2\,dx
+2\int_\Omega\phi_\rho|\mathbb Du|^2\,dx
+\int_{\partial\Omega}\alpha|u|^2\,d\sigma  \\
&+\int_{A_\rho}\mathbb T(u,p-p_\infty):(\overline{u}\otimes\nabla\phi_\rho)\,dx
-\int_\Omega (\eta\cdot\nabla u)\cdot\overline{u}\phi_\rho\,dx=0,
\end{split}
\label{cut-ene}
\end{equation}
where $A_\rho=\{\rho<|x|<2\rho\}$.
By taking into account
\[
\int_\Omega \Big(\big[(\eta\cdot\nabla u)\cdot\overline{u}+u\cdot[\eta\cdot\nabla\overline{u})\big]\phi_\rho+|u|^2\eta\cdot\nabla\phi_\rho\Big)\,dx
=\int_\Omega\mbox{div $(\phi_\rho\eta|u|^2)$}\,dx
=\int_{\partial\Omega}\eta\cdot\nu|u|^2\,d\sigma,
\]
the real and imaginary parts of \eqref{cut-ene} provide us with
\begin{equation}
\begin{split}
&(\mbox{Re $\lambda$})\int_\Omega\phi_\rho|u|^2\,dx
+2\int_\Omega\phi_\rho|\mathbb Du|^2\,dx
+\int_{\partial\Omega}\alpha|u|^2\,d\sigma  \\
&+\mbox{Re}\int_{A_\rho}\mathbb T(u,p-p_\infty):(\overline{u}\otimes\nabla\phi_\rho)\,dx
-\frac{1}{2}\int_{\partial\Omega}\eta\cdot\nu|u|^2\,d\sigma
+\frac{1}{2}\int_{A_\rho} |u|^2\eta\cdot\nabla\phi_\rho\,dx=0
\end{split}
\label{cut-ene-re}
\end{equation}
and
\begin{equation}
(\mbox{Im $\lambda$})\int_\Omega\phi_\rho|u|^2\,dx
+\mbox{Im}\int_{A_\rho}\mathbb T(u,p-p_\infty):(\overline{u}\otimes\nabla\phi_\rho)\,dx
-\mbox{Im}\int_\Omega (\eta\cdot\nabla u)\cdot\overline{u}\phi_\rho\,dx=0.
\label{cut-ene-im}
\end{equation}
Since we know from \eqref{improved-sum} that $\nabla u\in L^2(\Omega)$ as well as
\[
|\mathbb T(u,p-p_\infty)||u|+\frac{1}{2}|u|^2\eta\in L^{3/2}(\mathbb R^3\setminus B_\rho)
\]
and since $\|\nabla\phi_\rho\|_3$ is independent of $\rho$,
letting $\rho\to\infty$ in \eqref{cut-ene-re}--\eqref{cut-ene-im} gives
\begin{equation}
(\mbox{Re $\lambda$})\|u\|_2^2+2\|\mathbb Du\|_2^2+\int_{\partial\Omega}\left(\alpha-\frac{\eta\cdot\nu}{2}\right)|u|^2\,d\sigma=0,
\label{cut-ene-re2}
\end{equation}
\begin{equation}
(\mbox{Im $\lambda$})\|u\|_2^2-\mbox{Im}\int_\Omega(\eta\cdot\nabla u)\cdot\overline{u}\,dx=0.
\label{cut-ene-im2}
\end{equation}

Once we have those, 
as in the proof of the first half of Proposition \ref{int-uni} based on \eqref{ene-re}--\eqref{ene-im}, 
we employ \eqref{ass-fric} to get $u=0$ immediately for $\lambda\in \overline{\mathbb C_+}$ 
(resp. $\lambda\in\mathbb C\setminus (-\infty,0)$) when $\eta\in\mathbb R^3\setminus\{0\}$ (resp. $\eta=0$).
Unlike the case of bounded domains, $\mathbb Du=O$ (rigid motion) leads to $u=0$ on account of summability
\eqref{sum-weak} at infinity when $\mbox{Re $\lambda$}=0$.
For the case $\mbox{Re $\lambda$}<0$, we rely on Korn's inequality \eqref{korn} to deduce 
\[
(\mbox{Re $\lambda$})\|u\|_2^2+\frac{2}{c_0}\|\nabla u\|_2^2
+\int_{\partial\Omega}\left(\alpha-\frac{\eta\cdot\nu}{2}\right)|u|^2\,d\sigma\leq 0
\]
from \eqref{cut-ene-re2}.
This together with \eqref{cut-ene-im2} implies $u=0$ provided that $\lambda\in\mathbb C\setminus \widetilde S_\eta$
along the same line as in the proof of the second half of Proposition \ref{int-uni}.
Obviously, $p=p_\infty$ follows from $\nabla(p-p_\infty)=0$ and \eqref{improved-sum}.

Finally, with Lemma \ref{Du-wein} at hand,
the latter part under the assumption \eqref{ass-fric0} is verified as in the proof of Proposition \ref{int-uni}.
The proof is complete.
\end{proof}

We now construct a solution to \eqref{Os-resol}. 
Let $(E_\eta(\lambda),\Pi)$ be the solution operator \eqref{formula}--\eqref{formula-p}
in the whole space $\mathbb R^3$, and $(M_\eta(\lambda),N_\eta(\lambda))$
the one for \eqref{Os-int}--\eqref{ave-zero}
given by Proposition \ref{prop-int} with the specific bounded domains
\begin{equation}
D=\Omega_3=\Omega\cap B_3, \qquad 
D_0=A_1=\{1<|x|<2\}
\label{ext-cut}
\end{equation}
and the extended friction coefficient
\begin{equation}
\widetilde\alpha(x):=\left\{
\begin{array}{ll}
\alpha(x) \quad& (x\in \partial\Omega), \\
\alpha_0 &(x\in\partial B_3),
\end{array}
\right.
\qquad
\alpha_0>\frac{|\eta|}{2}+\frac{1}{3},
\label{new-fric}
\end{equation}
where the constant $\alpha_0$ is fixed by taking into account
the curvature of $\partial B_3$. 
Notice that 
the extended one \eqref{new-fric} actually fulfills
the assumption \eqref{ass-fric2} (resp. \eqref{ass-fric3})
for every $x\in \partial\Omega_3$
as long as \eqref{ass-fric} (resp. \eqref{ass-fric0})
is satisfied at $\partial\Omega$.
Even for the full slip case $\alpha\equiv 0$ at $\partial\Omega$, 
the extended one \eqref{new-fric}
is a positive friction at $\partial B_3$
and, hence, $\Omega_3$ (and thus the obstacle $\mathbb R^3\setminus\Omega$) is allowed to be axisymmetric along a certain axis.
To include an axisymmetric obstacle with the 
full slip condition, there is the other way adopted by Shimada and Yamaguchi \cite{SY08}, who replaced $B_3$ 
by a non-axisymmetric bounded domain $\widetilde B$ in \eqref{ext-cut}.

We take a cut-off function $\phi\in C_0^\infty(B_2;[0,1])$ such that $\phi=1$ in $B_1$, and use the Bogovskii operator
$\mathbb B$ in $A_1=\{1<|x|<2\}$ that is defined as follows.
The Dirichlet problem for the equation of continuity in a bounded domain $A_1$ admits a lot of solutions if the forcing term $g$ satisfies
the compatibility condition $\int_{A_1}g\,dx=0$.
Among those solutions, a particular one discovered by Bogovskii \cite{Bog79}
is useful, see also \cite{BoS90, G-b}; in fact, there is a linear operator $\mathbb B: C_0^\infty(A_1)\to C_0^\infty(A_1)^3$
such that, for $q\in (1,\infty)$ and integer $k\geq 0$,
\begin{equation}
\|\nabla^{k+1}\mathbb Bg\|_{q,A_1}\leq C\|\nabla^kg\|_{q,A_1}
\label{bog-est}
\end{equation}
with some $C=C(q,k)>0$, which is invariant with respect to dilation of the domain $A_1$, and that
\[
\mbox{div $(\mathbb Bg)$}=g\qquad\mbox{if}\;\int_{A_1}g(x)\,dx=0.
\]
By continuity, $\mathbb B$ extends to a bounded operator from $W^{k,q}_0(A_1)$ to $W^{k+1,q}_0(A_1)$.

For $R\geq 2$, we put
\begin{equation}
L^q_{[R]}(\Omega):=\{f\in L^q(\Omega);\; f(x)=0\;\mbox{a.e. $\Omega\setminus B_R$}\}.
\label{data-sp}
\end{equation}
Given $f\in L^q_{[R]}(\Omega)$ and $\lambda\in\overline{\mathbb C_+}$
(or $\lambda\in (\mathbb C\setminus S_\eta)\cup\{0\}$,
depending on the assumption, see Proposition \ref{prop-int}),
we set
\begin{equation}
\begin{split}
&v=R_\eta(\lambda)f:=(1-\phi)E_\eta(\lambda)f+\phi M_\eta(\lambda)f
+\mathbb B\left[\big(E_\eta(\lambda)f-M_\eta(\lambda)f\big)\cdot\nabla\phi\right],  \\
&\sigma=Q_\eta(\lambda)f:=(1-\phi)\Pi f+\phi \widetilde N_\eta(\lambda)f,
\end{split}
\label{parametrix}
\end{equation}
where $f$ is understood as its zero extension (resp. restriction) to $\mathbb R^3$ (resp. $\Omega_3$)
and the pressure $\widetilde N_\eta(\lambda)f$ in $\Omega_3$ is chosen in such a way that
\begin{equation}
\widetilde N_\eta(\lambda)f:=N_\eta(\lambda)f+\frac{1}{|A_1|}\int_{A_1}(\Pi f)(x)\,dx.
\label{int-press}
\end{equation}
Because of this choice, we have the Poincar\'e inequality
\begin{equation}
\|\Pi f-\widetilde N_\eta(\lambda)f\|_{q,A_1}\leq C\|\nabla (\Pi f-N_\eta(\lambda)f)\|_{q,A_1}.
\label{poin}
\end{equation}
The pair $(v,\sigma)$ obeys
\begin{equation}
\begin{array}{ll}
\lambda v-\Delta v-\eta\cdot\nabla v+\nabla\sigma=f+T_\eta(\lambda)f, \quad\mbox{div $v$}=0\;\;&\mbox{in $\Omega$}, \\
\nu\cdot v=0,\qquad [(2\mathbb Dv)\nu]_\tau+\alpha v_\tau=0 &\mbox{on $\partial\Omega$},
\end{array}
\label{ext-parame}
\end{equation}
and satisfies
\begin{equation}
v\in L^r(\Omega), \qquad \sigma\in L^s(\Omega), \qquad\nabla\sigma\in L^q(\Omega)
\label{para-sum}
\end{equation}
with some $r,\, s\in (1,\infty)$, where
\begin{equation}
\begin{split}
T_\eta(\lambda)f
&=2\nabla\phi\cdot\nabla(E_\eta(\lambda)f-M_\eta(\lambda)f)
+(\Delta\phi+\eta\cdot\nabla\phi)(E_\eta(\lambda)f-M_\eta(\lambda)f)  \\
&\quad -\Delta \mathbb B[(E_\eta(\lambda)f-M_\eta(\lambda)f)\cdot\nabla\phi]
+\lambda\mathbb B[(E_\eta(\lambda)f-M_\eta(\lambda)f)\cdot\nabla\phi]  \\
&\quad -\eta\cdot\nabla\mathbb B[(E_\eta(\lambda)f-M_\eta(\lambda)f)\cdot\nabla\phi]
-(\nabla\phi)(\Pi f-\widetilde N_\eta(\lambda)f).
\end{split}
\label{remain}
\end{equation}
In view of \eqref{parametrix},
the summability \eqref{para-sum} at infinity 
is determined by $\big(E_\eta(\lambda)f, \Pi f)$, which behaves like the fundamental solution 
at infinity since $f$ has a bounded support.
Concerning the 
velocity $v$, even in the worst case $(\lambda,\eta)=(0,0)$ as well as $q\in (1,3/2)$,
one can take $r\in (3,q_{**}]$, where $1/q_{**}=1/q-2/3$.
For the other cases, better summability (with smaller $r$) is available.
As for the pressure $\sigma$, one can take $s\in (3/2,q_*]$ (resp. $s\in (3/2,\infty)$) when 
$q\in (1,3)$ (resp. $q\in [3,\infty)$), where $1/q_*=1/q-1/3$.
Finally, $\nabla\sigma\in L^q(\Omega)$ follows from \eqref{fmt-p} and \eqref{high-int}.
\begin{proposition}
Suppose that a constant vector $\eta\in\mathbb R^3$ and a nonnegative function $\alpha\in C(\partial\Omega)$
fulfill the relation \eqref{ass-fric} for every $x\in\partial\Omega$.
Let $q\in (1,\infty)$, $R\geq 2$ and
\[
\begin{array}{ll}
\lambda\in\overline{\mathbb C_+}&\mbox{for $\eta\in\mathbb R^3\setminus\{0\}$}, \\
\lambda\in \mathbb C\setminus (-\infty,0)\qquad &\mbox{for $\eta=0$}.
\end{array}
\]
Then the operator $1+T_\eta(\lambda)$ is bijective on $L^q_{[R]}(\Omega)$ and
\begin{equation}
u=R_\eta(\lambda)(1+T_\eta(\lambda))^{-1}f, \qquad p=Q_\eta(\lambda)(1+T_\eta(\lambda))^{-1}f
\label{ext-sol}
\end{equation}
provides a solution to \eqref{Os-resol} for every $f\in L^q_{[R]}(\Omega)$, that is unique within the class specified in 
Proposition \ref{ext-uni},
where $L^q_{[R]}(\Omega)$ is given by \eqref{data-sp}.

For $x\in\partial\Omega$,
let $\kappa(x)\leq 0$ be the least eigenvalue of the Weingarten map $-\nabla N$ of $\partial\Omega$ in the direction of $\nu$
(see subsection \ref{wein}).
Suppose in addition that $\eta\in\mathbb R^3\setminus \{0\}$ and $\alpha\in C(\partial\Omega)$ fulfill the relation
\eqref{ass-fric0}
for every $x\in \partial\Omega$.
Then the same conclusion as above holds true for every
$\lambda\in (\mathbb C\setminus S_\eta)\cup\{0\}$,
where $S_\eta$ is given by \eqref{os-spec-0}.
\label{prop-parame}
\end{proposition}

\begin{proof}
From \eqref{fmt-p}, \eqref{bdd-wh},
\eqref{high-int} and \eqref{bog-est}
together with \eqref{poin}, the operator $T_\eta(\lambda)$ is bounded from $L^q_{[R]}(\Omega)$ into $W^{1,q}(\Omega)$.
Since $T_\eta(\lambda)f$ is compactly supported, $T_\eta(\lambda)$ is compact from $L^q_{[R]}(\Omega)$ into itself
by the Rellich theorem.

Let $f\in L^q_{[R]}(\Omega)$ satisfy $(1+T_\eta(\lambda))f=0$, from which we see that the support of $f$ is contained in 
$\overline{A_1}$.
By \eqref{bdd-wh}, Proposition \ref{prop-int}
and \eqref{para-sum}
one can apply Proposition \ref{ext-uni} to find that 
$v=R_\eta(\lambda)f=0$ and $\sigma=Q_\eta(\lambda)f=0$. 
In view of \eqref{parametrix}, we infer that
$(E_\eta(\lambda)f,\Pi f)=(0,0)$ for $|x|\geq 2$ and that $(M_\eta(\lambda)f,\widetilde N_\eta(\lambda)f)=(0,0)$ for $|x|\leq 1$.
Hence, both pairs are solutions to the Oseen resolvent system \eqref{Os-int} in $D=B_3$ with $f$ under consideration and with
the friction coefficient $\alpha_0>\frac{|\eta|}{2}+\frac{1}{3}$, see \eqref{new-fric}.
By \eqref{int-press} both pressures fulfill
\[
\int_{A_1}(\Pi f)(x)\,dx=\int_{A_1}(\widetilde N_\eta(\lambda)f)(x)\,dx.
\]
From Proposition \ref{int-uni} it follows that they must coincide with each other.
After all, $(E_\eta(\lambda)f,\Pi f)=(0,0)$ in the whole space $\mathbb R^3$and thereby $f=0$, that is, $1+T_\eta(\lambda)$
is injective.
It is thus bijective by the Fredholm alternative on $L^q_{[R]}(\Omega)$.
In this way, \eqref{ext-sol} is a solution to \eqref{Os-resol}.
The proof is complete.
\end{proof}

We rephrase Theorem \ref{thm-0} as the following proposition,
including the Stokes case $\eta=0$ shown by \cite{SS07},
together with a representation of the resolvent.
\begin{proposition}
Suppose that a constant vector $\eta\in\mathbb R^3$ and a nonnegative function $\alpha\in C(\partial\Omega)$
fulfill the relation \eqref{ass-fric} for every $x\in \partial\Omega$.
Let $q\in (1,\infty)$, then we have
\begin{equation}
\begin{array}{ll}
\overline{\mathbb C_+}\setminus \{0\}\subset \rho(-L) &\mbox{for $\eta\in\mathbb R^3\setminus\{0\}$}, \\
\mathbb C\setminus (-\infty,0]\subset \rho(-A)\qquad&\mbox{for $\eta=0$}.
\end{array}
\label{resol-set}
\end{equation}
If, in particular, $f\in L^q_{[R]}(\Omega)$ with $R\geq 2$, we have a representation of the resolvent
\begin{equation}
(\lambda+L)^{-1}Pf=R_\eta(\lambda)(1+T_\eta(\lambda))^{-1}f
\label{resol-para}
\end{equation}
for every $\lambda$ that belongs to the left-hand side of \eqref{resol-set},
where $L^q_{[R]}(\Omega)$ is given by \eqref{data-sp}.

If the assumption \eqref{ass-fric} is replaced by \eqref{ass-fric0} for $\eta\in \mathbb R^3\setminus \{0\}$,
then we have
\begin{equation}
\mathbb C\setminus S_\eta\subset \rho(-L)
\label{resol-set2}
\end{equation}
with $S_\eta$ given by \eqref{os-spec-0}.
\label{prop-resol}
\end{proposition}

\begin{proof}
It suffices to show \eqref{resol-set}, which at once yields \eqref{resol-para} for $f\in L^q_{[R]}(\Omega)$
by Proposition \ref{prop-parame}.
Let us take $\lambda$ from the left-hand side of \eqref{resol-set}.
We first verify that $\lambda+L$ is injective.
Let $u\in D(L)$ satisfy $(\lambda+L)u=0$ in $L^q_\sigma(\Omega)$, then one can take the associated pressure $p$
with $\nabla p\in L^q(\Omega)$ such that $(u,p)$ is a solution to \eqref{Os-resol} with $f=0$
and, therefore, $u=0$ by Proposition \ref{ext-uni}.

Given $f\in L^q(\Omega)$, we next construct a solution of \eqref{Os-resol}.
Set
\[
v=(1-\phi)E_\eta(\lambda)f+\mathbb B[(E_\eta(\lambda)f)\cdot\nabla\phi], \qquad
\sigma=(1-\phi)(\Pi f-\sigma_0)
\]
with $\sigma_0=|A_1|^{-1}\int_{A_1}\Pi f\,dx$, where
$(E_\eta(\lambda),\Pi)$ is the solution operator \eqref{formula}--\eqref{formula-p} in the whole space $\mathbb R^3$,
$f$ is understood as its zero extension to $\mathbb R^3$, $\phi$ is the same cut-off function as in \eqref{parametrix}, and
$\mathbb B$ is the Bogovskii operator in $A_1=\{1<|x|<2\}$.
We look for a solution to \eqref{Os-resol} of the form $u=v+w$ and $p=\sigma+\tau$.
Then $(w,\tau)$ obeys \eqref{Os-resol} with $f$ replaced by
\begin{equation*}
\begin{split}
g&=\phi f-2\nabla\phi\cdot\nabla E_\eta(\lambda)f-(\Delta \phi+\eta\cdot\nabla\phi)E_\eta(\lambda)f
+\Delta\mathbb B[(E_\eta(\lambda)f)\cdot\nabla\phi]  \\
&\quad -\lambda\mathbb B[(E_\eta(\lambda)f)\cdot\nabla\phi]
+\eta\cdot\nabla\mathbb B[(E_\eta(\lambda)f)\cdot\nabla\phi]+(\nabla\phi)(\Pi f-\sigma_0)
\end{split}
\end{equation*}
that belongs to $L^q_{[R]}(\Omega)$ with $R\geq 2$.
Thus Proposition \ref{prop-parame} provides a solution $(w,\tau)$ given by \eqref{ext-sol} with $f$ replaced by $g$.
By \eqref{fmt}
together with Proposition \ref{prop-int}
we see that $u=v+w\in D(L)$ (the case $\lambda=0$ is excluded here) together with $(\lambda+L)u=Pf$.
Hence, $\lambda+L$ is surjective and thereby invertible since $L$ is closed.

Finally, it is obvious that \eqref{resol-set2} follows from the latter half of Proposition \ref{prop-parame} in the argument above.
The proof is complete.
\end{proof}

In order to deduce the large time decay of the Oseen semigroup $e^{-tL}$, what is crucial is to investigate the behavior
of the resolvent \eqref{resol-para} near $\lambda=0$.
To this end, it suffices to consider $\lambda$ in the right-half plane including the imaginary axis.
We begin with the following lemma.
\begin{lemma}
Under the same assumptions of the first half of Proposition \ref{prop-parame},
let $q\in (1,\infty)$ and $R\geq 2$.
Given $m>0$ and compact set $K\subset \overline{\mathbb C_+}$,
there is a constant 
$c_*=c_*(m,K,\alpha,q,R,\Omega)>0$ such that
\begin{equation}
\big\|(1+T_\eta(\lambda))^{-1}\big\|_{{\mathcal L}(L^q_{[R]}(\Omega))}\leq c_*
\label{remain-unif}
\end{equation}
for all $(\lambda,\eta) \in K\times \overline{B_m}$, where $L^q_{[R]}(\Omega)$ is given by \eqref{data-sp}.
\label{lem-unif}
\end{lemma}

\begin{proof}
In view of \eqref{remain}, we find from 
\eqref{wh-conti}, \eqref{conti-int} and \eqref{bog-est} that
$(\lambda,\eta)\mapsto T_\eta(\lambda)\in {\mathcal L}(L^q_{[R]}(\Omega))$ is continuous on 
$\overline{\mathbb C_+}\times \mathbb R^3$; hence, so is
$(\lambda,\eta)\mapsto (1+T_\eta(\lambda))^{-1}\in {\mathcal L}(L^q_{[R]}(\Omega))$ by the argument using the Neumann series 
as in the proof of Proposition \ref{prop-int}.
As a consequence,
\eqref{remain-unif} holds true. 
The proof is complete.
\end{proof}

The next proposition tells us the regularity of the resolvent along the imaginary axis with respect to the topology
${\mathcal L}(L^q_{[R]}(\Omega),\,W^{2,q}(\Omega_3))$ and plays a key role.
\begin{proposition}
Under the same assumptions of the first half of 
Proposition \ref{prop-parame}, let $q\in (1,\infty)$ and $R\geq 2$.
We set
\begin{equation}
V(\tau): f\mapsto
\partial_\tau(i\tau+L)^{-1}Pf,
\qquad \tau\in\mathbb R\setminus\{0\}. 
\label{on-ima}
\end{equation}
For every $m>0$, 
there is a constant $C=C(m,\alpha,q,R,\Omega)>0$ such that
\begin{equation}
\int_{-4}^4
\|V(\tau)\|_{{\mathcal L}(L^q_{[R]}(\Omega),W^{2,q}(\Omega_3))}
\,d\tau\leq C 
\label{key-1}
\end{equation}
\begin{equation}
\int_{-2}^2
\|V(\tau+h)-V(\tau)\|_{{\mathcal L}(L^q_{[R]}(\Omega),W^{2,q}(\Omega_3))}
\,d\tau\leq C|h|^{1/2} 
\label{key-2}
\end{equation}
for all $\eta\in\mathbb R^3$ with $|\eta|\leq m$  
and $h\in\mathbb R$ with $|h|\leq 1$,
where $L^q_{[R]}(\Omega)$ is given by \eqref{data-sp}.
\label{ext-key}
\end{proposition}

\begin{proof}
In what follows 
we make use of \eqref{remain-unif} with
$K=\{\lambda\in\overline{\mathbb C_+};\; |\lambda|\leq 4\}$. 
The proof is based on the structure of the parametrix of the resolvent, see \eqref{parametrix}, \eqref{remain} and \eqref{resol-para}. 
Let us split 
$V(\tau)$ into
\begin{equation*}
V(\tau)=V_1(\tau)+V_2(\tau)
\end{equation*}
with
\begin{equation}
\begin{split}
&
V_1(\tau)
=[\partial_\tau R_\eta(i\tau)](1+T_\eta(i\tau))^{-1}, \\
&
V_2(\tau)
=-R_\eta(i\tau)(1+T_\eta(i\tau))^{-1}[\partial_\tau T_\eta(i\tau)](1+T_\eta(i\tau))^{-1}.
\end{split}
\label{rep-1st}
\end{equation}
For simplicity of notation, we set
\[
F(\tau)=\partial_\tau R_\eta(i\tau), \qquad
G(\tau)=(1+T_\eta(i\tau))^{-1}, \qquad
H(\tau)=\partial_\tau T_\eta(i\tau). 
\]
The obvious equality
\begin{equation}
G(\tau+h)-G(\tau)
=-G(\tau+h)
\big[T_\eta(i(\tau+h))-T_\eta(i\tau)\big]
G(\tau)
\label{rep-inv}
\end{equation}
is useful in the computations below.

We now collect \eqref{wh-1}--\eqref{wh-4}, \eqref{high-int} with $0\leq j\leq 2$, \eqref{bog-est}, \eqref{int-press},
\eqref{remain-unif} and \eqref{rep-inv} to find that
\begin{equation}
\sup_{|\tau|\leq 4}\|
R_\eta(i\tau)
\|_{{\mathcal L}(L^q_{[R]}(\Omega),W^{2,q}(\Omega_3))}\leq C,
\label{ext-reg1}
\end{equation}
\begin{equation}
\int_{-4}^4\left(
\|F(\tau)\|_{{\mathcal L}(L^q_{[R]}(\Omega),W^{2,q}(\Omega_3))}+
\|H(\tau)\|_{{\mathcal L}(L^q_{[R]}(\Omega))}
\right)\,d\tau\leq C,
\label{ext-reg2}
\end{equation}
yielding \eqref{key-1}, and that
\begin{equation}
\sup_{|\tau|\leq 4}\left(
\|
R_\eta(i(\tau+h))-R_\eta(i\tau)
\|_{{\mathcal L}(L^q_{[R]}(\Omega),W^{2,q}(\Omega_3))}+
\|G(\tau+h)-G(\tau)\|_{{\mathcal L}(L^q_{[R]}(\Omega))}
\right)
\leq C|h|^{1/2},
\label{ext-reg3}
\end{equation}
\begin{equation}
\int_{-2}^2\left(
\|F(\tau+h)-F(\tau)\|_{{\mathcal L}(L^q_{[R]}(\Omega),W^{2,q}(\Omega_3))}+
\|H(\tau+h)-H(\tau)\|_{{\mathcal L}(L^q_{[R]}(\Omega))}
\right)\,d\tau
\leq C|h|^{1/2}.
\label{ext-reg4}
\end{equation}
In view of \eqref{rep-1st},
we employ \eqref{ext-reg1}--\eqref{ext-reg4} to deduce
\begin{equation*}
\begin{split}
&\quad\int_{-2}^2
\|V_1(\tau+h)-V_1(\tau)\|_{{\mathcal L}(L^q_{[R]}(\Omega),W^{2,q}(\Omega_3))}
\,d\tau  \\
&\leq c_* 
\int_{-2}^2\|F(\tau+h)-F(\tau)\|_{{\mathcal L}(L^q_{[R]}(\Omega), W^{2,q}(\Omega_3))}\,d\tau  \\
&\quad +\sup_{|\tau|\leq 2}\|G(\tau+h)-G(\tau)\|_{{\mathcal L}(L^q_{[R]}(\Omega))} 
\int_{-2}^2\|F(\tau)\|_{{\mathcal L}(L^q_{[R]}(\Omega_3), W^{2,q}(\Omega_3))}\,d\tau  \\
&\leq C|h|^{1/2} 
\end{split}
\end{equation*}
as well as
\begin{equation*}
\begin{split}
&\quad\int_{-2}^2
\|V_2(\tau+h)-V_2(\tau)\|_{{\mathcal L}(L^q_{[R]}(\Omega),W^{2,q}(\Omega_3))}
\,d\tau  \\
&\leq c_*^2\sup_{|\tau|\leq 2}\|
R_\eta(i(\tau+h))-R_\eta(i\tau)
\|_{{\mathcal L}(L^q_{[R]}(\Omega),W^{2,q}(\Omega_3))} 
\int_{-2}^2 \|H(\tau+h)\|_{{\mathcal L}(L^q_{[R]}(\Omega))}\,d\tau  \\
&\quad +c_*\sup_{|\tau|\leq 2}\|
R_\eta(i\tau)
\|_{{\mathcal L}(L^q_{[R]}(\Omega),W^{2,q}(\Omega_3))} 
\left[c_*\int_{-2}^2\|H(\tau+h)-H(\tau)\|_{{\mathcal L}(L^q_{[R]}(\Omega))}\,d\tau \right. \\
&\left. 
\quad +\sup_{|\tau|\leq 2}\|G(\tau+h)-G(\tau)\|_{{\mathcal L}(L^q_{[R]}(\Omega))}
\int_{-2}^2 \Big(\|H(\tau+h)\|_{{\mathcal L}(L^q_{[R]}(\Omega))}
+\|H(\tau)\|_{{\mathcal L}(L^q_{[R]}(\Omega))}\Big)\,d\tau \right]  \\
&\leq C|h|^{1/2} 
\end{split}
\end{equation*}
which completes the proof.
\end{proof}

\section{Proof of Theorem \ref{thm}}
\label{proof}

In this section we prove Theorem \ref{thm}.
To this end, the following elementary lemma is useful.
This provides us with the relation between the regularity of functions and the rate of decay of those (inverse) Fourier transform.
\begin{lemma}
[{\cite[Lemma 7.3]{H16}}]
Let $X$ be a Banach space with norm $\|\cdot\|$ and $w\in L^1(\mathbb R;X)$.
Then
the function
\begin{equation}
W(t)
=\int_{-\infty}^\infty e^{i\tau t}w(\tau)\,d\tau
\label{fourier}
\end{equation}
enjoys
\begin{equation}
\|
W(t)
\|\leq C\int_{-\infty}^\infty \left\|w\left(\tau+\frac{1}{t}\right)-w(\tau)\right\|\,d\tau
\label{est-four}
\end{equation}
for all $t\in\mathbb R\setminus \{0\}$.
\label{lem-fourier}
\end{lemma}

The idea of the proof of Lemma \ref{lem-fourier} is found first in \cite{H04}, however, the origin of this lemma
goes back to Shibata \cite{S83}.
Estimate of the form \eqref{est-four} would be easier to apply to various situations.

We observe
that \eqref{key-1} enables us to justify the formula 
\begin{equation}
e^{-tL}Pf=\frac{-1}{2\pi it}\int_{-\infty}^\infty e^{i\tau t}\partial_\tau (i\tau +L)^{-1}Pf\,d\tau
\label{repre}
\end{equation}
in 
$W^{1,q}(\Omega_3)$ as long as $f\in L^q_{[R]}(\Omega)$ when we start from the Dunford integral representation, 
in which the path of integration may hit the origin on account of \eqref{wh-1} and \eqref{high-int}
together with \eqref{remain-unif},
perform 
integration by parts and then move the path 
to the imaginary axis.

Let us close the paper with 
the following proposition on the 
local energy decay estimate.
In fact, once we have that,
then we arrive at Theorem \ref{thm}
along the same way as in the existing literature, see
\cite{I89, KS98, ES05, H04, H16, H20, H21,HS09,Ma21,S08,SY08}, 
by means of 
well-established cut-off procedure which consists in two steps:
(i) estimate still in $\Omega_3$ for general $f\in L^q_\sigma(\Omega)$; 
(ii) estimate in $\mathbb R^3\setminus\Omega_3$.
Since the procedure would be rather standard nowadays,
one may omit the proof of Theorem \ref{thm}.
For the Stokes semigroup ($\eta=0$), the following proposition was already proved 
by Shimada and Yamaguchi \cite[Theorem 1.5]{SY08}.
\begin{proposition}
Suppose that a constant vector $\eta\in\mathbb R^3$ and a nonnegative function
$\alpha\in C^1(\partial\Omega)$ fulfill the relation \eqref{ass-fric} for every $x\in\partial\Omega$.
Let $q\in (1,\infty)$ and $R\geq 2$.
For every $m>0$, there is a constant $C=C(m,\alpha,q,R,\Omega)>0$ such that
\begin{equation}
\|e^{-tL}Pf\|_{W^{2,q}(\Omega_3)}\leq Ct^{-3/2}\|f\|_q
\label{led}
\end{equation}
for all $t\geq 1$, $\eta\in\mathbb R^3$ with $|\eta|\leq m$ and $f\in L^q_{[R]}(\Omega)$ which is given by \eqref{data-sp}.
\label{prop-led}
\end{proposition}

\begin{proof}
We take a cut-off function $\psi\in C^\infty(\mathbb R;[0,1])$ such that $\psi(\tau)=1$ for $|\tau|\leq 1$ and
$\psi(\tau)=0$ for $|\tau|\geq 2$.
Let us divide the integral \eqref{repre} into
\begin{equation}
\frac{-1}{2\pi it}\int_{-\infty}^\infty e^{i\tau t}\psi(\tau)\partial_\tau (i\tau +L)^{-1}Pf\,d\tau
\label{repre-main}
\end{equation}
and the other part which decays like $t^{-2}$ by integration by parts once more since
\[
\|\partial_\tau^2(i\tau +L)^{-1}Pf\|_{W^{2,q}(\Omega)}
=2\|(i\tau+L)^{-3}Pf\|_{W^{2,q}(\Omega)}\leq C|\tau|^{-2}\|f\|_q
\]
for $|\tau|\geq 1$ and $f\in L^q(\Omega)$.
It thus suffices to deduce the decay rate $t^{-1/2}$ of 
$W(t)$
of the form \eqref{fourier} with
\[
w(\tau)=\psi(\tau)v(\tau), \qquad v(\tau)=\partial_\tau (i\tau+L)^{-1}Pf.
\]
Let $f\in L^q_{[R]}(\Omega)$.
From \eqref{est-four} in Lemma \ref{lem-fourier} with $X=W^{2,q}(\Omega_3)$ we know
\begin{equation*}
\begin{split}
\|
W(t)
\|_{W^{2,q}(\Omega_3)}
&\leq C\int_{-\infty}^\infty \left|\psi\left(\tau+\frac{1}{t}\right)-\psi(\tau)\right|
\left\|v\left(\tau+\frac{1}{t}\right)\right\|_{W^{2,q}(\Omega_3)}\,d\tau  \\
&\qquad +C\int_{-\infty}^\infty |\psi(\tau)|\left\|v\left(\tau+\frac{1}{t}\right)-v(\tau)\right\|_{W^{2,q}(\Omega_3)}\,d\tau
=:I+J
\end{split}
\end{equation*}
for $t\in\mathbb R\setminus \{0\}$.
If, in particular, $t\geq 1$, then we readily see from \eqref{key-1} and
$|\psi(\tau+1/t)-\psi(\tau)|\leq C/t$
that
\[
I\leq \frac{C}{t}\int_{|\tau|\leq 3}\left\|v\left(\tau+\frac{1}{t}\right)\right\|_{W^{2,q}(\Omega_3)}\,d\tau
\leq Ct^{-1}\|f\|_q.
\]
In addition, it follows from \eqref{key-2} that
\[
J\leq C\int_{|\tau|\leq 2}\left\|v\left(\tau+\frac{1}{t}\right)-v(\tau)\right\|_{W^{2,q}(\Omega_3)}\,d\tau\leq Ct^{-1/2}\|f\|_q
\]
which completes the proof of \eqref{led}.
\end{proof}

\noindent
{\bf Acknowledgments}.
The author would like to thank Professor Bum Ja Jin for stimulating discussions about her paper \cite{J25}.
He is partially supported by the Grant-in-aid for Scientific Research 25K07083 from JSPS.

\noindent
{\bf Declarations}.
The author states that there is no conflict of interest.
Data sharing not applicable to this article as no datasets were generated or analyzed during the current study.

\end{document}